\documentclass[a4paper]{amsart}
\usepackage{graphicx}
\usepackage{amssymb,mathtools}
\usepackage{mathrsfs}
\usepackage{stmaryrd}
\usepackage[mathcal]{euscript}
\usepackage{tikz}
\usepackage{tikz-cd}
\usetikzlibrary{decorations.pathreplacing}

\usepackage{tcolorbox}

\usepackage{hyperref}
\hypersetup{%
  bookmarksnumbered=true,%
  colorlinks=true,%
  linkcolor=blue,%
  citecolor=blue,%
  filecolor=blue,%
  menucolor=blue,%
  urlcolor=blue,%
  bookmarksopen=true,%
  bookmarksdepth=2,%
  pageanchor=true}

\title[Infinite length modules over tame hereditary algebras]{From finite to infinite length modules\\ over tame hereditary algebras}

\author{Lidia Angeleri H\"ugel}
\address{Dipartimento di Informatica - Settore di Matematica\\
Università degli Studi di Verona\\
I-37134 Verona\\ Italy}
\email{lidia.angeleri@univr.it}

\author{Andrew Hubery}
\address{Fakult\"at f\"ur Mathematik\\
Universit\"at Bielefeld\\ D-33501 Bielefeld\\ Germany}
\email{hubery@math.uni-bielefeld.de}

\author{Henning Krause}
\address{Fakult\"at f\"ur Mathematik\\
Universit\"at Bielefeld\\ D-33501 Bielefeld\\ Germany}
\email{hkrause@math.uni-bielefeld.de}

\theoremstyle{plain}
\newtheorem{thm}{Theorem}[section]
\newtheorem{prop}[thm]{Proposition}
\newtheorem{lem}[thm]{Lemma} 

\newtheorem{cor}[thm]{Corollary}

\theoremstyle{definition}
\newtheorem{defn}[thm]{Definition}
\newtheorem{exm}[thm]{Example}

\theoremstyle{remark}
\newtheorem{rem}[thm]{Remark}

\numberwithin{equation}{section}

\newcommand{\add}{\operatorname{add}}
\newcommand{\Add}{\operatorname{Add}}

\newcommand{\Coh}{\operatorname{Coh}}

\newcommand{\Coker}{\operatorname{Coker}}
\newcommand{\Coprod}{\operatorname{Coprod}}
\newcommand{\colim}{\operatorname*{colim}}

\newcommand{\End}{\operatorname{End}}

\newcommand{\Ext}{\operatorname{Ext}}

\newcommand{\fp}{\operatorname{fp}}

\newcommand{\Hom}{\operatorname{Hom}}

\renewcommand{\Im}{\operatorname{Im}}

\newcommand{\Ker}{\operatorname{Ker}}

\renewcommand{\mod}{\operatorname{mod}}
\newcommand{\Mod}{\operatorname{Mod}}

\newcommand{\PExt}{\operatorname{PExt}}

\newcommand{\Prod}{\operatorname{Prod}}

\newcommand{\Tor}{\operatorname{Tor}}

\newcommand{\Ab}{\mathrm{Ab}}

\newcommand{\op}{\mathrm{op}}

\newcommand{\character}{\raisebox{1pt}[0pt][2pt]{$\chi$}}

\newcommand{\iso}{\xrightarrow{\raisebox{-.4ex}[0ex][0ex]{$\scriptstyle{\sim}$}}}

\newcommand*{\intref}[2]{\def\tmp{#1}\ifx\tmp\empty\hyperref[#2]{\ref*{#2}}\else\hyperref[#2]{#1~\ref*{#2}}\fi}

\renewenvironment{quote}{%
  \list{}{%
    \leftmargin5ex   
    \rightmargin\leftmargin
  }
  \item\relax
}
{\endlist}
     
\def\A{\mathcal A} 
 
\def\C{\mathcal C}
\def\D{\mathcal D} 
\def\E{\mathcal E} 
\def\F{\mathcal F}

\def\I{\mathcal I}

\def\P{\mathcal P}
\def\R{\mathcal R}
 
\def\T{\mathcal T} 
\def\U{\mathcal U}

\def\X{\mathcal X}
\def\Y{\mathcal Y}

\def\bfL{\mathbf L}

\def\bfL{\mathbf L}

\def\bbN{\mathbb N}

\def\bbX{\mathbb X} 
 
\def\bbZ{\mathbb Z}

\newcommand{\frt}{\mathfrak t}
\newcommand{\frf}{\mathfrak f}
\newcommand{\fri}{\mathfrak i}

\def\La{\Lambda}
\def\Si{\Sigma}

\begin{document}

\date{\today}

\begin{abstract}
  A self-contained introduction to infinite dimensional
  representations over a tame hereditary algebra is provided, assuming
  a basic knowledge of the category of finite dimensional
  representations. This includes a complete description of all
  pure-injective modules.  Of particular interest are the torsionfree
  divisible modules, which are precisely the direct sums of copies of
  the unique generic module.
\end{abstract}

\dedicatory{Dedicated to Claus Michael Ringel on the occasion of his
  80th birthday.}


\maketitle 

\section{Introduction}

Beyond finite representation type, a systematic study of
infinite dimensional representations over finite dimensional algebras was
initiated by Claus Michael Ringel in his seminal `Rome notes'
published in 1979 \cite{Ri1979}. We quote from its introduction:
\begin{quote}
  Why is it of interest to consider infinite dimensional
  representations? We believe that the only reason for the usual
  restriction in dealing with finite dimensional algebras to consider
  only modules of finite length, is the fact that modules of finite
  length are easier to handle.
\end{quote}
 One of the main achievements in Ringel's work is to single out a
particular module for a tame hereditary algebra; it is the unique
indecomposable representation which is of infinite length but of
finite endolength.  Following Crawley-Boevey, such modules are now
called `generic' \cite{CB1992}. Let us see how Ringel introduces this
module in his work:
\begin{quote}
  A typical example would be the Kronecker module
  $(k[X], k[X], \text{id}, \cdot X)$. Of course, in dealing with
  finitely generated modules over a general ring, the injective
  envelopes of the finitely generated modules are important. They are
  usually no longer finitely generated, but are for certain types of
  rings, for example noetherian rings, well behaved. In a suitable
  subcategory, the injective envelope of the Kronecker module $(k[X],
  k[X], \text{id}, \cdot X)$ is just the Kronecker module $Q = (k(X),
  k(X), \text{id}, \cdot X)$,
  with $k(X)$ being the field of rational functions in one variable. We
  will see that this Kronecker module plays a dominant role, it will
  be characterized as the unique indecomposable `torsionfree
  divisible' module. Thus, there always will be certain important
  infinite dimensional representations, and we will see that the
  investigation of these modules also gives some new insight into the
  behaviour of the modules of finite length.
\end{quote}

Almost 50 years later, some further techniques have been developed,
and the purpose of this note is to give from this perspective a rather
self-contained introduction to infinite dimensional representations
over a tame hereditary algebra. We assume a basic knowledge about
the category of finite dimensional representations and concentrate on
techniques that help to understand the class of pure-injective
modules. A full classification of these modules is possible when the
algebra is tame hereditary, and this class includes the torsionfree
divisible modules.

Most of the results about modules over tame hereditary
algebras which are presented here are due to Ringel \cite{Ri1979}. In the last section we
include some further material and give appropriate references. In
particular, we provide a complete description of all pure-injective
modules; while the result is known, our proof seems to be new. For unexplained concepts
and general facts about modules of infinite length we refer to
\cite{CB2,JL1989,Kr2022,KR2000,Pr2009}.

\section{Preliminaries}

Let $k$ be a field, and $\Lambda$ a finite dimensional $k$-algebra
which is tame hereditary and connected. We study the category
$\Mod\Lambda$ of all right $\Lambda$-modules, and denote by
$\mod\Lambda$ the full subcategory of finite dimensional
$\Lambda$-modules. Left $\Lambda$-modules are identified with right
modules over the opposite algebra $\Lambda^\op$. For the definition of
a tame hereditary algebra and the structure of $\mod\Lambda$ we refer
to \cite{DR1976}.

The vector space dual $D=\Hom_k(-,k)$ yields for a $\Lambda$-module
$X$ and a $\Lambda^\op$-module $Y$  natural isomorphisms
\begin{equation}\label{eq:k-dual}
  \Hom_\Lambda(X,DY) \cong D(X\otimes_\Lambda Y) \cong
  \Hom_{\Lambda^\op}(Y,DX).
\end{equation}  

For an additive category $\C$ and full additive subcategories $\C_1,\C_2$ we write
\[ \C = \C_1 \vee \C_2 \]
provided $\C_1\cap\C_2=\{0\}$ and every object $C\in\C$ admits a
decomposition $C=C_1\oplus C_2$ with $C_i\in\C_i$.

Given a full additive subcategory $\C\subseteq\mod\Lambda$ we set
\[ \bar\C \coloneqq \{ X \in \Mod\Lambda \mid X = \colim_{i\in I} C_i \text{ with all } C_i \in \C \}, \]
so the completion of $\C$ under filtered colimits. Also, let
\[ \C^\perp \coloneqq \{ X \in \Mod\Lambda \mid
  \Hom_\Lambda(C,X)=0=\Ext^1_\Lambda(C,X) \text{ for all } C\in\C
  \} \] be the right perpendicular category.

\subsection*{Torsion pairs}

Recall that a pair of full subcategories  $(\T,\F)$ of an abelian
category $\A$ is a \emph{torsion pair} if 
there are equalities
\[ \T =\{ X \in \A\mid \Hom(X,Y)=0\text{ for all } Y\in\F \} \]
\[ \F =\{ Y \in \A\mid \Hom(X,Y)=0\text{ for all } X\in\T \}\]
and each  $X\in\A$ fits into an exact sequence
$0 \to \frt(X) \to X \to \frf(X) \to 0$ with $\frt(X)\in\T$ and 
$\frf(X)\in\F$. In this case $\frt$ provides a right adjoint for the
inclusion $\T \to\A$ and $\frf$ provides a left adjoint for the
inclusion $\F \to\A$.

\begin{lem}
  Let $(\T,\F)$ be a torsion pair for $\mod\Lambda$. Then
  $(\bar\T,\bar\F)$ is a torsion pair for $\Mod\Lambda$, and
  $\bar\F=\{ Y \in \Mod\Lambda \mid \Hom_\Lambda(X,Y)=0 \text{ for all
  } X\in\T \}$.
\end{lem}

\begin{proof}
For $X=\colim_i X_i$ in $\bar\T$ and $Y=\colim_j Y_j$ in $\bar\F$ we have
\[ \Hom_\Lambda(X,Y) = \lim\nolimits_i \colim\nolimits_j \Hom_\Lambda(X_i,Y_j) = 0. \]
For an arbitrary module $X=\colim_{i\in I} X_i$, written as filtered colimit of finite dimensional modules, we have the canonical exact sequences
\[ 0 \to \frt(X_i) \to X_i \to \frf(X_i) \to 0 \qquad (i\in I) \]
with $\frt(X_i)\in\T$ and $\frf(X_i)\in\F$. Taking their filtered
colimit provides an exact sequence with end terms in $\bar\T$ and
$\bar\F$, respectively. The last assertion is clear since
$\Hom_\Lambda(-,Y)$ takes filtered colimits to limits. 
\end{proof}

\begin{rem}
  Suppose that $(\T,\F)$ is a split torsion pair for $\mod\La$. Then
  any exact sequence $0\to X'\to X\to X''\to 0$ with
  $X'\in\bar\T$ and $X''\in\bar\F$ is \emph{pure exact}, so a filtered
  colimit of split exact sequences. Moreover such a
  sequence splits whenever there are only finitely many indecomposables
  either in $\T$ or in $\F$; see Corollary~\ref{co:endofinite}.
\end{rem}

Our starting point is the decomposition
\[ \mod\Lambda = \P \vee \R \vee \I \] where $\P$, $\R$ and $\I$ are
the full subcategories of \emph{preprojective}, \emph{regular}, and
\emph{preinjective} modules, respectively. These subcategories are
defined in terms of the Auslander--Reiten translate; see \cite{DR1976}
and further below for more details. Note that $(\R\vee\I,\P)$ is a
split torsion pair for $\mod\Lambda$, as is $(\I,\P\vee\R)$. It
follows that $(\overline{\R\vee\I},\bar\P)$ is a torsion pair for
$\Mod\Lambda$, and so every module $X$ admits a pure exact sequence
\[ 0 \to \frt(X) \to X \to \frf(X) \to 0. \]

\begin{defn}
Let $X$ be a $\Lambda$-module. Then we say:
\begin{enumerate}
\item $X$ is \emph{preinjective} if it is in $\bar\I$
\item $X$ is \emph{torsion} if it is in $\overline{\R\vee\I}$
\item $X$ is \emph{torsionfree} if it is in $\bar\P$
\end{enumerate}
As every finite dimensional preinjective module is a quotient of a regular module, we see that $X$ is torsionfree if and only if $\Hom_\Lambda(\R,X)=0$. We further define
\begin{enumerate}
\setcounter{enumi}{3}
\item $X$ is \emph{divisible} if $\Hom_\Lambda(X,\R)=0$
\item $X$ is \emph{regular} if $\Hom_\Lambda(X,\P)=0=\Hom_\Lambda(\I,X)$
\end{enumerate}
A module $X$ is divisible if and only if $\Ext^1_\Lambda(\R,X)=0$. This
follows from the Auslander--Reiten formula.
\end{defn}

\subsection*{The Auslander--Reiten translate}

Because $\La$ is hereditary the bimodule
$\tau^-\Lambda\coloneqq\Ext^1(D\Lambda,\Lambda)$ yields the adjoint
pair of functors \[\tau^-\coloneqq -\otimes_\Lambda\tau^-\Lambda\qquad\text{and}
\qquad\tau\coloneqq\Hom_\Lambda(\tau^-\Lambda,-)\] on $\Mod\Lambda$
extending the usual Auslander--Reiten translate on $\mod\Lambda$. A
finite dimensional indecomposable module $X$ is then preprojective if
and only if $\tau^nX=0$ for some $n>0$, and preinjective if and only
if $\tau^{-n}X=0$ for some $n>0$.

As $\Lambda$ is tame, one can furthermore define the \emph{defect} $\partial(X)$ of a finite dimensional module $X$. This is an additive function taking integer values such that, if $X$ is indecomposable, then $X$ is preprojective if and only if $\partial(X)<0$, and preinjective if and only if $\partial(X)>0$. We further assume that $\partial$ is normalised, so has range $\bbZ$.

\subsection*{Preinjective modules}

We begin with a description of the preinjective modules because their
structure is relatively simple.

\begin{prop}\label{pr:preinj}
  Every module $X$ admits a maximal preinjective submodule $\fri(X)$,
  which is furthermore a pure submodule and a direct sum of
  indecomposable direct summands of $X$ lying in $\I$.
\end{prop}
\begin{proof}
  We follow \cite[\S3]{Ri1979}.  Using the split torsion pair
  $(\I,\P\vee\R)$ on $\mod\Lambda$ we see that every module $X$ admits
  a maximal preinjective submodule $\fri(X)$, which is furthermore a
  pure submodule of $X$. Explicitly, we proceed inductively as
  follows.  For $n\ge 0$ let
  $\I_n=(\Ker\tau^{-n})\cap\mod\Lambda$. This is a torsion class in
  $\mod\La$ and we write $\fri_n(X)$ for the maximal submodule of $X$
  in $\bar\I_n=\Ker\tau^{-n}$, which is a direct summand of
  $X$.  Here one uses that $\I_n$ contains up to isomorphism only finitely many
  indecomposables, and therefore $\bar\I_n$ consists of the
  direct sums of copies of the indecomposables in $\I_n$; cf.\ Corollary~\ref{co:endofinite}.  We obtain a
  chain of split monomorphisms
  $\fri_0(X)\to \fri_1(X)\to \fri_2(X)\to\cdots$ such that
  \[\fri(X)=\bigcup_{n\ge 0}\fri_n(X)=\bigoplus_{n\ge
    0}(\fri_{n+1}/\fri_n)(X)\] where $(\fri_{n+1}/\fri_n)(X)$ is a
  direct sum of copies of the finitely many indecomposables in
  $\I_{n+1}-\I_n$.
\end{proof}

\section{Purity and endofiniteness}\label{se:purity}

In this section we collect some of the concepts that are used to study
modules of infinite length. Of particular interest are pure-injective
modules, with important examples being the endofinite modules. We
discuss the decomposition properties of such modules and introduce
the useful notion of a definable subcategory. Much of this material is
based on work of Crawley-Boevey \cite{CB1992,CB2} and actually independent
from the fact that $\La$ is tame hereditary.

\subsection*{Purity}

Recall that an exact sequence $0\to X\to Y\to Z\to 0$ in $\Mod\La$ is
\emph{pure exact} if it stays exact after applying $-\otimes_\La M$
for any $\La^\op$-module $M$. An equivalent condition is that the
sequence is a filtered colimit of split exact sequences. In this case
$X\to Y$ is called \emph{pure monomorphism}, and $X$ is
\emph{pure-injective} if every pure monomorphism $X\to Y$ splits. The
module $X$  is \emph{$\Si$-pure injective} if every direct sum of copies
of $X$ is pure-injective, and $X$ is \emph{endofinite} if it is of finite
length as an $\End_\La(X)$-module.

\subsection*{Coherent functors} 

Recall that an additive functor $F\colon\Mod\Lambda\to\Ab$ into the
category of abelian groups is \emph{coherent} if it admits a presentation
\[ \Hom_\Lambda(C_1,-) \xrightarrow{(\phi,-)} \Hom_\Lambda(C_0,-) \longrightarrow F \longrightarrow 0 \]
with $\phi\colon C_0 \to C_1$ in $\mod\Lambda$. We note that $\Ext^1_\Lambda(C,-)$ and $-\otimes_\Lambda C'$ are coherent for all $C\in\mod\Lambda$ and $C'\in\mod\Lambda^\op$.

The following lemma collects some basic properties which will be used freely.

\begin{lem}\label{le:coherent}
The coherent functors form a full abelian subcategory of the
category of all additive functors $\Mod\Lambda\to\Ab$. 
Moreover, for $F$ coherent the following hold.
\begin{enumerate}
\item The functor $F$ preserves filtered colimits and products.
\item The functor $F$ maps pure exact sequences to exact sequences.
\item The assignment  \[ \Coker\Hom_\Lambda(\phi,-) = F \longmapsto F^\vee\coloneqq \Ker(\phi\otimes_\Lambda-) \]
induces a contravariant equivalence between the coherent functors for
$\Lambda$ and those for $\Lambda^\op$. Moreover, there is a natural
isomorphism $F\cong
(F^\vee)^\vee$. 
\item  For $X\in\Mod\Lambda$ and  $Y\in\Mod\La^\op$ there are natural isomorphisms
  \[F(D^2X)\cong D^2F(X)\qquad\text{and}\qquad DF^\vee(Y)\cong F(DY).\] 
\item The functor $F$ is simple if and only if $F=\Coker\Hom_\Lambda(\phi,-)$ for some left almost split morphism $\phi$ in $\mod\Lambda$.
\item A $\La$-module $X$ has no finite length indecomposable direct summands if and only if $S(X)= 0$ for each simple coherent functor $S$.
\end{enumerate}
\end{lem}  

\begin{proof}
  It is clear that coherent functors are closed under taking
  cokernels. Using the fact that $\mod\Lambda$ has cokernels, one
  shows that coherent functors are also closed under taking
  kernels. This yields the first assertion. Notice that the category
  of coherent functors has enough projective objects. This follows
  from the fact that the representable functors are projective by
  Yoneda's Lemma.
  
(1) This is clear when $F=\Hom_\Lambda(C,-)$, and the general case follows since taking cokernels commutes with filtered colimits and products.

(2) Every pure exact sequence is a filtered colimit of split exact sequences.

(3) We write $F$ as the cokernel of $\Hom_\Lambda(\phi,-)$ for some
$\phi\colon C_0\to C_1$. Then $F^\vee$ is coherent because it is the
kernel of a morphism between coherent functors. Given a coherent
functor $G$ for $\La^\op$ we claim that there is a natural isomorphism
\[\Hom(F,G^\vee)\cong\Hom(G,F^\vee).\]
This is clear when $F=\Hom_\La(C,-)$ and the general case follows
using the exactness of $0\to F^\vee\to \Hom_\La(C_0,-)^\vee\to
\Hom_\La(C_1,-)^\vee$. The unit $F\to
(F^\vee)^\vee$ of the adjunction is an isomorphism. Again, this is
clear when $F$ is representable and the general case follows by exactness.

(4) The adjunction
\[ D(\phi\otimes_\Lambda Y) \cong \Hom_\Lambda(\phi,DY) \] for
$Y\in\Mod\Lambda^\op$ gives the isomorphism $DF^\vee(Y)\cong F(DY)$.
Applying this isomorphism twice and using the isomorphism
$F\iso (F^\vee)^\vee$ yields, for $X\in\Mod\Lambda$, the isomorphism
\[ D^2F(X) \cong F(D^2X). \]

(5) See \cite[II.2]{Au1978}. This uses the existence of left almost split morphisms for modules over a finite dimensional $k$-algebra.

(6) follows from (5).
\end{proof}

\subsection*{Definable subcategories}

A full subcategory $\C\subseteq\Mod\Lambda$ is \emph{definable} if there is a family of coherent functors $(F_i)_{i\in I}$ such that
\[ \C = \{ X \in \Mod\Lambda \mid F_i(X)=0 \text{ for all } i\in I\}. \]
It follows immediately from Lemma~\ref{le:coherent}(4) that $\C$ is then closed under taking double duals. Also, every $X\in\C$ admits a pure embedding $X\to D^2X$ where $D^2X\in\C$ is pure-injective. In particular, $\C$ is closed under pure-injective envelopes.

\begin{exm}\label{ex:definable}
(1) Let $(\T,\F)$ be a torsion pair for $\mod\Lambda$. Then
\[ \bar\F = \{ X\in\Mod\Lambda \mid \Hom_\Lambda(\mathcal T,X) = 0 \} \]
is definable. Specific examples are the torsionfree modules $\bar\P$ and the modules in $\overline{\P\vee\R}$.

(2) For $C\in\mod\Lambda$ there is a natural isomorphism
\[\Hom_\Lambda(-,C)\cong \Hom_\Lambda(-,D^2C)\cong D(-\otimes_\La DC)\] given by adjunction, and therefore
$\{X\in\Mod\Lambda\mid\Hom_\Lambda(X,C)=0\}$ is definable. Choosing
for $C$ the modules in $\mathcal R$ or in $\mathcal P$, we deduce that
the divisible modules and regular modules both form definable
subcategories.

(3) Let $\C\subseteq\mod\Lambda$. Then $\C^\perp$ is definable.

(4) The $\Lambda$-modules having no finite length indecomposable
direct summands form a definable subcategory, see
Lemma~\ref{le:coherent}(6).
\end{exm}

\subsection*{Subgroups of finite definition}

Let $X$ be a $\Lambda$-module. Each morphism $C\to C'$ in $\mod\Lambda$ yields a map $\Hom_\Lambda(C',X)\to\Hom_\Lambda(C,X)$, whose image is called a \emph{subgroup of finite definition} (or H-subgroup, see \cite{zbcn}) of $\Hom_\Lambda(C,X)$. Note that each such subgroup is an $\End_\Lambda(X)$-submodule, and the subgroups of finite definition form a modular lattice $\bfL_C(X)$. Taking $C=\Lambda$ gives the lattice $\bfL(X)$ of subgroups of finite definition of $X$.

On the other hand, the coherent functors form an abelian category by
Lemma~\ref{le:coherent}, so the coherent subfunctors of a coherent
functor $F$ form a modular lattice $\bfL_F$; see \cite[IV,
Proposition~5.3]{St1975}.
Evaluation at $X$ then induces a surjective
lattice homomorphism
\[\bfL_F\twoheadrightarrow\bfL_F(X)\coloneqq\{F'(X)\mid F'\subseteq F\text{ coherent}\}.\]
In particular $\bfL_C(X)=\bfL_F(X)$ for $F=\Hom_\La(C,-)$.

We observe that, for any set $I$, there are natural identifications
$\bfL_C(X^{(I)})=\bfL_C(X)=\bfL_C(X^I)$. Taking a subspace
$U\rightarrowtail X$ to the kernel of $DX\twoheadrightarrow DU$
induces a lattice isomorphism $\bfL(X)^\op\iso\bfL(DX)$. Also, for a
coherent functor $F$ and a coherent subfunctor $G\subseteq F$ the
lattice $\bfL_G(X)$ identifies with the interval $[0, G(X)]$ in
$\bfL_F(X)$, while  $\bfL_{F/G}(X)$ identifies with the interval $[G(X),F(X)]$. 

Chain conditions on subgroups of finite definition provide effective characterisations of $\Sigma$-pure injective and endofinite modules. See \cite{GrusonJensen}, \cite{Z}, \cite[Theorem~3.1 and Proposition~4.1]{CB1992}.

\begin{prop}\label{pr:subgroup-fin-def}
\begin{enumerate}
\item A module is $\Sigma$-pure injective if and only if it has the descending chain condition (dcc) on subgroups of finite definition.
\item A module $X$ is endofinite if and only if it has both acc and dcc on subgroups of finite definition, in which case $\bfL_C(X)$ identifies with the lattice of submodules of $\Hom_\Lambda(C,X)$.
\end{enumerate}
\end{prop}

\begin{proof}
  We begin by observing that the chain conditions hold for $\bfL(X)$
  if and only if they hold for all $\bfL_C(X)$. One direction is
  clear, taking $C=\Lambda$. For the converse observe that any
  coherent functor $G$ is a subquotient of a finite direct sum of
  copies of the forgetful functor $F=\Hom_\La(\La,-)$.
  Thus it follows from the above description of lattices of subgroups
  of finite definition given by extensions of coherent functors that any chain
  condition for $\bfL(X)=\bfL_F(X)$ implies the condition for
  $\bfL_G(X)$. It remains to note that  $\bfL_C(X)=\bfL_G(X)$ for $G=\Hom_\La(C,-)$. 

(1) Assume $X$ has dcc on subgroups of definition. Then so too does $X^{(I)}$, since $\bfL(X^{(I)})=\bfL(X)$, and it is enough to show that $X$ is pure-injective.

Take a pure exact sequence $0\to X\to Y\to Z\to 0$ and write $Z=\colim Z_\alpha$ with $Z_\alpha\in\mod\Lambda$. This yields the inverse system of short exact sequences
\[ 0 \to \Hom_\Lambda(Z_\alpha,X) \to \Hom_\Lambda(Z_\alpha,Y) \to \Hom_\Lambda(Z_\alpha,Z) \to 0. \]
Moreover, the condition on $\bfL_{Z_\alpha}(X)$ shows that the Mittag-Leffler condition holds, so taking limits yields the exact sequence
\[ 0 \to \Hom_\Lambda(Z,X) \to \Hom_\Lambda(Z,Y) \to \Hom_\Lambda(Z,Z) \to 0. \]
It follows that the original sequence splits, and hence that $X$ is pure-injective.

Conversely assume that $X$ is $\Sigma$-pure injective. Then the canonical inclusion $X^{(\bbN)}\to X^\bbN$ admits a left inverse $\pi$. Assume for contradiction that we have a strict chain $\Hom_\Lambda(C,X)=U_0\supset U_1\supset\cdots$ in $\bfL_C(X)$ and for each $i$ take $f_i\in U_i-U_{i+1}$. This yields an element $f=(f_i)\in\Hom_\Lambda(C,X^\bbN)$. Take $n$ such that $(\pi f)_r=0$ for all $r\geq n$, and decompose $f=f'+f''$ having supports $\leq n$ and $>n$ respectively. Then $\pi f=f'+\pi f''$, so $f_n=f'_n=-(\pi f'')_n$.

Now, $f''$ lies in $U_{n+1}^\bbN$, which is a subgroup of finite definition in $\Hom_\Lambda(C,X^\bbN)$, and hence an $\End_\Lambda(X^\bbN)$-submodule. Thus $\pi f''$ must lie in $U_{n+1}^\bbN$, so $f_n=-(\pi f'')_n$ lies in $U_{n+1}$, a contradiction.
\medskip

(2) As all subgroups of finite definition of $X$ are $\End_\Lambda(X)$-submodules, we see that $X$ being endofinite implies $\bfL_C(X)$ has both acc and dcc.

For the converse assume first that $X$ is pure-injective. Write $X=\colim X_\alpha$ with $X_\alpha\in\mod\Lambda$ and associated morphisms $j_\alpha\colon X_\alpha\to X$. Given $\phi\colon C\to X$ with $C\in\mod\Lambda$, we can find $\beta$ such that $\phi=j_\beta\phi_\beta$. We claim that, inside $\Hom_\Lambda(C,X)$, we have
\[ \End_\Lambda(X)\phi = \bigcap_{\alpha\geq\beta}\Hom_\Lambda(X_\beta,X)\phi_\alpha. \]
One inclusion is clear, so suppose $\psi$ lies in the intersection. We factor $\phi=\bar\phi\pi$ and $\phi_\alpha=\bar\phi_\alpha\pi_\alpha$ via their respective images $\bar C$ and $\bar C_\alpha$, yielding commutative diagrams
\[ \begin{tikzcd}
\varepsilon_\alpha \colon &[-20pt] 0 \arrow[r] & \bar C_\alpha \arrow[r,"\bar\phi_\alpha"] \arrow[d,"i_\alpha"] &
X_\alpha \arrow[r] \arrow[d,"j_\alpha"] & Y_\alpha \arrow[r] \arrow[d] & 0\\
\varepsilon \colon &[-20pt] 0 \arrow[r] & \bar C \arrow[r,"\bar\phi"] & X \arrow[r] & Y \arrow[r] & 0
\end{tikzcd} \]
Moreover, $\varepsilon=\colim\varepsilon_\alpha$, which agrees with the colimit of the pushout sequences $\colim i_\alpha\varepsilon_\alpha$.

By assumption we have $\psi=\psi_\alpha\bar\phi_\alpha\pi_\alpha$. Using $\pi=i_\alpha\pi_\alpha$ we see that $\Ker\pi=\colim\Ker\pi_\alpha$, and hence $\psi$ factors as $\tilde\psi\pi$. We deduce that $\psi_\alpha\bar\phi_\alpha=\bar\psi i_\alpha$. Finally, the pushout $\bar\psi\varepsilon$ is the colimit of $\bar\psi i_\alpha\varepsilon_\alpha=\psi_\alpha\bar\phi_\alpha\varepsilon_\alpha$. The latter are all split exact, so $\bar\psi\varepsilon$ is pure, and hence also split as $X$ is pure-injective. Thus $\bar\psi=\gamma\bar\phi$ for some $\gamma\in\End_\Lambda(X)$, and $\psi=\gamma\phi$ factors through $\phi$.

If now $X$ is $\Sigma$-pure injective, then it has dcc on subgroups of finite definition, so each cyclic $\End_\Lambda(X)$-submodule of $\Hom_\Lambda(C,X)$ is a subgroup of finite definition. If we also assume acc, then the same holds for every $\End_\Lambda(X)$-submodule, showing that $X$ is endofinite.
\end{proof}

\begin{rem}\label{re:quot-lattice}
If $X'$ belongs to the definable subcategory generated by $X$, then the canonical map $\bfL_C\twoheadrightarrow\bfL_C(X')$ factors though $\bfL_C\twoheadrightarrow\bfL_C(X)$, so induces a surjective lattice homomorphism $\bfL_C(X)\twoheadrightarrow\bfL_C(X')$. It follows that $\bfL_C(X')$ inherits any chain condition from $\bfL_C(X)$.

In particular this applies to pure submodules $X'\subseteq X$, in which case the morphism can be described explicitly as
\[ \bfL_C(X) \to \bfL_C(X'), \quad F(X) \mapsto F(X')=F(X)\cap\Hom_\Lambda(C,X'). \]
As a consequence, if $X$ is $\Sigma$-pure injective, then so too is $X'$, and the inclusion $X'\subseteq X$ splits.
\end{rem}

\begin{lem}\label{le:end-generic}
Let $X$ be an indecomposable pure-injective module. Then $\End_\Lambda(X)$ is local. If moreover $X$ is endofinite, then $\End_\Lambda(X)$ has nilpotent radical, and this vanishes whenever $\Ext^1_\Lambda(X,X)=0$.
\end{lem}

\begin{proof}
Set $E\coloneqq\End_\Lambda(X)$. The assignment sending $X$ to $T_X=X\otimes_\Lambda-$ determines a functor from $\Mod\Lambda$ to the category $\Add(\mod\Lambda^\op,\Ab)$ of additive functors. This is fully faithful and sends (indecomposable) pure-injective modules to (indecomposable) injective functors. Thus $E$ is isomorphic to $\End(T_X)$, which is local when $X$ is indecomposable pure-injective.

Now suppose that $X$ is indecomposable and endofinite. Then $E$ has nilpotent radical by Nakayama's Lemma. Assume further that $X$ has no self-extensions. Given a non-injective $f\in E$, we can factor it via its image $I$ as $f=ip$. Set $C\coloneqq\Coker f$ and $\xi\in\Ext^1_\Lambda(C,I)$ the corresponding extension. As $\Lambda$ is hereditary and $X$ has no self-extensions, there are surjections $\Hom(I,X)\twoheadrightarrow\Ext^1_\Lambda(C,X)\twoheadrightarrow\Ext^1_\Lambda(C,I)$. In particular, there exists $j\colon I\to X$ such that $\xi=pj\xi$. Now $jp\in E$ is again non-injective, so nilpotent, and hence $\xi=0$. As $X$ is indecomposable, $I=0$. An analogous argument holds if $f$ is not surjective.
\end{proof}

\begin{cor}\label{co:decomposition}
Every $\Sigma$-pure injective $\Lambda$-module $X$ admits an essentially unique decomposition $X=\bigoplus_i X_i$ into indecomposable modules with local endomorphism rings.
\end{cor}

\begin{proof}
The argument is a straightforward application of Zorn's Lemma. First observe that there is an indecomposable direct summand when $X\neq 0$; see Lemma~\ref{le:simply-reflexive}. Now choose a maximal family of indecomposable pure-injective submodules $(X_i)_{i\in I}$ such that $X'=\sum_iX_i$ is a pure submodule of $X$ and the sum is direct. Then $X'$ is a direct summand by Remark~\ref{re:quot-lattice}. Thus $X'=X$ by our first observation.
\end{proof}

\begin{cor}\label{co:endofinite}
Let $X$ be an endofinite $\Lambda$-module. Then all direct sums of indecomposable direct summands of $X$ form a definable subcategory. When $X$ is of finite length this equals $\bar\X$ where $\X\subseteq\mod\La$ is the additive closure of $X$.
\end{cor}

\begin{proof}
See \cite[Theorem~4.2]{KS1998} or \cite[Corollary 13.1.13]{Kr2022}. By Proposition~\ref{pr:subgroup-fin-def} and Remark~\ref{re:quot-lattice} each module in the definable subcategory $\C$ generated by $X$ is endofinite, so decomposes into a direct sum of indecomposables by Corollary~\ref{co:decomposition}. Each such indecomposable direct summand is then isomorphic to a direct summand of $X$ by Lemma~\ref{le:simply-reflexive-endofinite}. This yields the first assertion.

Now let $X\in\mod\Lambda$ and let $\X$ denote its additive closure. From the first part we see that $\bar\X$ contains $\C$; on the other hand, every definable subcategory is closed under filtered colimits, so $\C$ contains $\bar\X$.
\end{proof}

\section{Torsion regular and divisible modules}\label{se:tor-reg}

In this section we look at torsion regular and divisible
$\La$-modules. In particular we discuss when these modules are
pure-injective. As a consequence, we see that there is a unique
generic $\La$-module up to isomorphism. In fact, we show that the
torsionfree divisible modules are precisely the direct sums of copies
of the generic module.

\subsection*{Torsion regular modules}

The category $\R$ is an abelian subcategory of $\mod\Lambda$ which is
furthermore \emph{uniserial}, so each indecomposable object admits a
unique composition series in $\R$. Moreover, the ext-quiver of $\R$ is
a disjoint union of (infinitely many) oriented cycles. Writing
$\mathbb X$ for the indexing set of the cycles in the ext-quiver of
$\R$, we have that \[\R=\coprod_{x\in\mathbb X}\R_x\] meaning that
every module in $\R$ is a direct sum $M=\bigoplus_xM_x$ of modules
$M_x\in\R_x$ such that $M_x=0$ for almost all $x$ and
$\Hom_\La(M_x,M_y)=0$ for all $x\neq y$.

\begin{lem}
The category of torsion regular
modules equals the completion $\bar\R$.
\end{lem}
\begin{proof}
  It is clear that all modules in $\bar\R$ are torsion
  regular. Conversely, let $M$ be torsion regular. As $M$ is torsion
  the quotient $M/\fri(M)$ belongs to $\bar\R$. On the other hand,
  $\fri(M)=0$ by Proposition~\ref{pr:preinj} since $M$ is regular.
\end{proof}

The category $\bar\R$ is a Grothendieck category since  $\R$ is abelian,
and we remark that $\bar\R$ is a proper subcategory
of all regular modules.  In fact, as we will see, the generic module is
regular but not in $\bar\R$.

For each simple object $S\in\R$ there is a chain of monomorphisms
\[ S = S_1 \rightarrowtail S_2 \rightarrowtail S_3 \rightarrowtail
  \cdots \] in $\R$ such that $S_n$ is indecomposable and of length
$n$ in $\R$. The \emph{Pr\"ufer module} with regular socle $S$ is by
definition
\[ \bar S := \colim_{n\geq 1} S_n. \] An application of Baer's
criterion shows that $\bar S$ is an injective object in $\bar\R$, so
an injective envelope of $S$, and therefore $\End_\Lambda(\bar S)$ is
local.

\begin{lem}\label{le:torsion-indec}
  Every non-zero torsion module admits an indecomposable direct
  summand, which is either of finite length or a Pr\"ufer module.
\end{lem}
\begin{proof}
  Let $M\neq 0$ be torsion. If $\fri(M)\neq 0$, then $\fri(M)$ admits
  a finite length indecomposable direct summand by
  Proposition~\ref{pr:preinj}, which is also a summand of $M$ since
  $\fri(M)$ is a pure submodule of $M$. Otherwise, $\fri(M)= 0$ and
  $M$ lies in $\bar\R$.  Then the assertion follows from the fact that
  $\R$ is uniserial; see \cite[Theorem~13.1.28.]{Kr2022}.
\end{proof}

For each simple object $S\in\R$ there exists $n\ge 1$ such that
$\tau^nS\cong S$. Then the cokernel of the inclusion
$S_n \rightarrowtail S_{in}$ is isomorphic to $S_{(i-1)n}$ for all
$i\ge 1$. The family of epimorphisms
$\phi_i\colon S_{in}\twoheadrightarrow S_{(i-1)n}$ induces an
epimorphism $\phi\colon\bar S\twoheadrightarrow\bar S$ such that
$\phi=\bigcup_{i\ge 1}\phi_i$ and $\Ker\phi^i=S_{in}$ for all
$i\ge 1$.
  
\begin{lem}\label{le:prufer}
  Each Pr\"ufer module is artinian over its endomorphism ring, and
  hence $\Sigma$-pure injective over $\Lambda$. Moreover, the
  injective objects in $\bar\R$ are precisely the coproducts of
  Pr\"ufer modules.
\end{lem}

\begin{proof}
  Consider $\bar S$ as a module over the noetherian local ring
  $k[t]_{(t)}$ via the homomorphism $k[t]\to \End_\La(\bar S)$ that
  sends $t$ to $\phi$. As such $\bar S$ is an essential extension of
  $S_n=\Ker\phi$. The latter is finite dimensional, so $\bar S$ embeds into
  a finite direct sum of copies of the injective envelope of $k$, which is
  artinian by \cite[Proposition~3]{Ma1960}. Thus $\bar S$ is artinian
  over $\End_\La(\bar S)$, and therefore $\Sigma$-pure injective by
  Proposition~\ref{pr:subgroup-fin-def}.

  The category $\bar\R$ is a \emph{locally noetherian Grothendieck category}, so
  it has a generating set of noetherian objects. Therefore the
  injective objects are precisely the coproducts of the indecomposable
  injective objects, which are the injective envelopes of the simple
  objects \cite[IV.2]{Ga1962}.
\end{proof}

The following result summarises the structure of the torsion regular
modules, which follows from the fact that $\R$ is a uniserial
category, and its decomposition into  connected components.

\begin{prop}\label{pr:tor-reg}
  Every torsion regular module admits a unique decomposition
  $M=\bigoplus_{x\in\mathbb X}M_x$ where $M_x\in\bar\R_x$. Moreover,
  there are pure exact sequences $0\to L_x\to M_x\to N_x\to 0$ where
  $L_x$ is a direct sum of modules in ${\R_x}$ and $N_x$ is a direct
  sum of Pr\"ufer modules in $\bar\R_x$.
\end{prop}
\begin{proof}
  Given $M\in\bar\R$, we have $M=\bigoplus_xM_x$, where
  $M_x\in\bar\R_x$ is the colimit over all submodules $U\subseteq M$
  with $U\in\R_x$. Now fix $x\in\bbX$ and let $M\in\bar\R_x$.  For
  $n\ge 0$ let $\R_{x,n}$ denote the subcategory of objects in $\R_x$ of
  Loewy length at most $n$; it has up to isomorphism only finitely many indecomposable
  objects.  Let $M_n$ be the maximal pure submodule of $M$ which is of
  Loewy length at most $n$; it is a direct sum of indecomposables
  in $\R_{x,n}$ and therefore a direct summand. We obtain a chain of
  split monomorphisms $M_0\to M_1\to M_2\to\cdots$ and set
  \[L\coloneqq\bigcup_{n\ge 0}M_n=\bigoplus_{n\ge 0}M_{n+1}/M_n\]
  where $M_{n+1}/M_n$ is a direct sum of copies of the finitely many
  indecomposables in $\R_{x,n+1}-\R_{x,n}$. This yields a pure exact
  sequence $0\to L\to M\to N\to 0$ in $\bar\R_x$. Recall from
  Lemma~\ref{le:torsion-indec} that every non-zero module in
  $\bar\R_x$ must have a direct summand which is either in ${\R_x}$ or a
  Pr\"ufer module. By construction $N$ has no direct summand in ${\R_x}$;
  thus $N$ is a direct sum of Pr\"ufer modules. Here, one uses that
  every object has a maximal injective summand since $\bar\R_x$ is
  locally noetherian.
\end{proof}

We note every torsion module is the direct sum of a preinjective
module and a torsion regular module; see
Proposition~\ref{pr:preinjectives-split}. On the other hand, if $X$ is
regular, then both $\frt(X)$ and $\frf(X)$ are regular, so
$\frt(X)\in\bar\R$.

\subsection*{Torsionfree and divisible modules}

We investigate the category of divisible modules and show that they
are $\Si$-pure injective. Moreover, the torsionfree divisible modules
are precisely the direct sums of copies of a distinguished
indecomposable  endofinite module.

The $k$-duality for $\Lambda$-modules \eqref{eq:k-dual} implies for a $\Lambda$-module $X$:
\begin{align*}
X \text{ divisible} \qquad &\iff \qquad DX \text{ torsionfree}\\
X \text{ torsionfree} \qquad &\iff \qquad DX\text{ divisible}.
\end{align*}
The first equivalence is clear. As $\Hom_\Lambda(R,-)$ is coherent for
all $R\in\R$, the module $X$ is torsionfree if and only if $D^2X$ is
torsionfree by Lemma~\ref{le:coherent}. Then the isomorphism
$\Hom_\Lambda(R,D^2X)\cong \Hom_{\Lambda^\op}(DX,DR)$ yields the
second equivalence.

We observe that the subcategory of torsionfree divisible modules
identifies with $\R^\perp$. Also, every torsionfree divisible module is
regular, and has no non-zero finite dimensional direct summands.

The following is a  key result for our discussion.
\begin{lem}\label{le:sub-P}
Let $P$ be a projective $\Lambda$-module of defect $-1$ and $X$ a regular $\Lambda$-module.
\begin{enumerate}
\item If $\frf(X)$ is non-zero, then $\Hom_\La(P,\frf(X))\neq 0$, and
  every non-zero map $P\to\frf(X)$ is injective with regular cokernel.
\item There is a submodule $P_X\subseteq X$ such that $P_X$ is a direct sum of copies of $P$ and $X/P_X$ is torsion regular.
\end{enumerate}
\end{lem}

\begin{proof}
(1) The subcategory of $\Lambda$-modules $M$ satisfying $\Hom_\Lambda(P,M)=0$ identifies with the module category of a proper factor algebra of $\Lambda$, which is therefore necessarily of finite representation type. Thus if $\Hom_\Lambda(P,\frf(X))=0$, then $\frf(X)$ is a direct sum of finite dimensional preprojective modules, contradicting that $X$ is regular. It follows that there exists a non-zero morphism $\phi\colon P\to\frf(X)$.

Next $\partial(\Im\phi)<0$ as $\frf(X)$ is torsionfree, so $\Ker\phi=0$. Also, $\Hom_\Lambda(\frf(X)/P,\P)=0$ as $X$ is regular. Finally, if we have $P\subset U\subset\frf(X)$ with $U/P\in\I$, then $\partial(U)=\partial(U/P)-1\geq0$, contradicting that $\frf(X)$ is torsionfree. Thus $\frf(X)/P$ is regular.

(2) Let $\U$ denote the set of all submodules $U\subseteq X$ that are direct sums of copies of $P$, and such that $X/U$ is regular. Note that both direct sums of copies of $P$ and regular modules are closed under taking filtered colimits. Thus every chain in $\U$ has an upper bound, and so $\U$ has a maximal element $P_X$ by Zorn's Lemma.

We claim that $X/P_X$ is torsion. If not, then by (1) there is an
inclusion $P\rightarrowtail\frf(X/P_X)$ with regular cokernel $C$. We
lift this to an inclusion $P\rightarrowtail X$ and obtain an exact
sequence $0\to \frt(X/P_X)\to X/(P+P_X)\to C\to 0$. Both end terms are regular. Thus $X/(P+P_X)$ is also regular, contradicting the maximality of $P_X$.
\end{proof}

\begin{lem}\label{le:ext-divisible-reg}
Let $X,Y$ be $\Lambda$-modules such that $X$ is regular and $Y$ is divisible. Then $\Ext^1_\Lambda(X,Y)=0$. In particular,  every divisible module in $\bar{\R}$ is a direct sum of Pr\"ufer modules.
\end{lem}

\begin{proof}
  Let $P$ be projective of defect $-1$. By Lemma~\ref{le:sub-P} there
  is a short exact sequence $0\to P_X\to X\to X'\to0$ such that
  $P_X$  is a direct sum of copies of $P$ and $X'$ is torsion regular. Now
  $\Ext^1_\Lambda(P_X,Y)$ vanishes as $P$ is projective.

We are reduced to considering the case when $X$ is torsion regular. By
Lemma~\ref{le:prufer} $X$ embeds into a direct sum of Pr\"ufer
modules, so it is enough to prove the result when $X=\bar S$ is itself
a Pr\"ufer module. We have $\Ext^1_\La(\bar S,Y)=0$, since
$\bar S=\bigcup_{n\ge 1}S_n$ and $\Ext^1_\La(S_n/S_{n-1},Y)=0$ for all
$n\ge 1$; cf.\ \cite[Lemma 6.2]{GT}.
  
Finally, we see that every divisible $Y\in\bar\R$ is an injective object, so a direct
sum of Pr\"ufer modules by Lemma~\ref{le:prufer}.
\end{proof}   

\begin{prop}\label{pr:divisible-pure-inj}
  Every divisible $\La$-module is $\Si$-pure injective.
\end{prop}
\begin{proof}
  Let $X$ be divisible. Since the divisible modules form a definable
  subcategory, there is a pure exact sequence
  $\xi\colon 0\to X\to Y\to Z\to 0$ where $Y$ is pure-injective and
  all terms are divisible (take $Y=D^2X$). Recall from Proposition~\ref{pr:preinj} that
  $\fri(Z)$ is a direct sum of finite length modules and therefore
  pure-projective. Hence the pullback of $\xi$ along the embedding
  $\fri(Z)\rightarrowtail Z$ splits. It follows that $\xi$ is the pullback of an
  exact sequence in $\Ext^1_\Lambda(Z/\fri(Z),X)$, which vanishes by
  Lemma~\ref{le:ext-divisible-reg} as $Z/\fri(Z)$ is
  regular. We conclude that $\xi$ splits, and so
  $X$ is pure-injective. Given that every direct sum of copies of $X$
  is divisible, it follows that $X$  is $\Si$-pure injective.
  \end{proof}

\begin{cor}\label{co:ext-divisible}
   If $X,Y$ are $\La$-modules such that $\fri(X)=0$ and $Y$ is
   divisible, then $\Ext^1_\La(X,Y)=0$.
   \end{cor}
 \begin{proof}
   Let $\xi\colon 0\to Y\to E\to X\to 0$ be an exact sequence. The
   pullback along the embedding $X'\rightarrowtail X$ of a finite
   length submodule splits, because in this case $X'\in \P\vee\R$, and
   every module from $\P$ embeds into a module from $\R$. It follows
   that $\xi$ is pure exact, and therefore $\xi$ splits since $Y$ is
   pure-injective by Proposition~\ref{pr:divisible-pure-inj}.
 \end{proof}
 
 \begin{thm}\label{th:endofinite}
Every torsionfree divisible $\Lambda$-module is endofinite, and all indecomposable torsionfree divisible $\Lambda$-modules are isomorphic. 
\end{thm}

\begin{proof}
Let $X$ be a torsionfree divisible $\Lambda$-module. By Proposition~\ref{pr:subgroup-fin-def} we need to show that
$\bfL(X)$ satisfies both acc and dcc. The module $X$ is $\Sigma$-pure injective by Proposition~\ref{pr:divisible-pure-inj}, so $\bfL(X)$ has dcc. Similarly the dual $DX$ is torsionfree divisible over $\Lambda^\op$, so $\bfL(X)\cong\bfL(DX)^\op$ has acc and we are done.

Let $X$ and $Y$ be indecomposable torsionfree divisible $\Lambda$-modules. Take a projective module $P$ of defect $-1$. Then Lemma~\ref{le:sub-P} gives monos $P\rightarrowtail X$ and $P\rightarrowtail Y$ with regular cokernels. Their pushout yields monos $X\rightarrowtail E$ and $Y\rightarrowtail E$, which split by Lemma~\ref{le:ext-divisible-reg}. We obtain maps $\phi\colon X\to Y$ and $\psi\colon Y\to X$ such that $\psi\phi$ is the identity on $P$, and therefore an isomorphism since
$\End_\Lambda(X)$ is a division ring by Lemma~\ref{le:end-generic}. The same holds for $\phi\psi$, and hence $X\cong Y$.
\end{proof}

The converse also holds provided the endofinite module has no finite length direct summands.

\begin{thm}\label{th:endofinite2}
Every endofinite $\Lambda$-module without finite length indecomposable direct summands is torsionfree divisible.
\end{thm}

\begin{proof}
Let $X$ be such an endofinite module. The torsion part $\frt(X)$ is a pure submodule, so endofinite and a direct summand of $X$ by Remark~\ref{re:quot-lattice}. If $\frt(X)$ is non-zero, then it admits an indecomposable direct summand which is either finite dimensional or a Prüfer module. As the Prüfer module is not endofinite we deduce that $\frt(X)=0$, so $X$ is torsionfree.

Next, the dual $DX$ is again endofinite, by Proposition~\ref{pr:subgroup-fin-def}. If $U$ is a finite dimensional direct summand of $DX$, then $DU$ is a finite dimensional direct summand of $D^2X$. The latter lies in the definable subcategory generated by $X$, so by Example ~\ref{ex:definable}(4) we must have $U=0$. By the first part we deduce that $DX$ is torsionfree, and hence that $X$ is divisible.
\end{proof}

\section{The generic module}

A module is called \emph{generic} provided it is indecomposable, endofinite, and of infinite length. From Theorems~\ref{th:endofinite} and \ref{th:endofinite2} we know that such a module over $\Lambda$ must be torsionfree divisible, and unique up to isomorphism.
  
\subsection*{Two constructions}

We offer two constructions of torsionfree divisible modules. 

Fix an indecomposable projective $\Lambda$-module $P$ of defect $-1$. For each simple object $S\in\R$ and $r\geq1$ there is a minimal universal coextension
\[ 0 \to P\to E_r \to S_r^{n} \to 0, \]
so with $\Hom_\Lambda(S_r,S_r^{n})\iso\Ext^1_\Lambda(S_r,P)$, where $n$ equals the dimension of $\Ext^1_\Lambda(S,P)$ over $\End_\Lambda(S)$. These form a filtered system, so their colimit is an exact sequence
\[ \xi_S \colon 0 \to P\to E \to \bar S^{n} \to 0. \]
Note that each $E_r$ is preprojective, so $E$ is torsionfree. The family of these sequences $\xi_S$, where $S$ runs through a representative set of simple objects in $\R$, determines an exact sequence
\[ \xi \colon 0 \to P \to Q \to \bigoplus\nolimits_S\bar S^n \to 0. \]
The following is a consequence of the construction of this sequence.

\begin{lem}\label{le:generic-construction}
The $\Lambda$-module $Q$ is generic.
\end{lem}

\begin{proof}
See \cite[5.2]{Ri1979}.  For each regular simple $S$ the exact sequence $\xi_S$ determines an isomorphism $\Hom_\Lambda(S,\bar S^{n})\iso\Ext^1_\Lambda(S,P)$. The same then holds for $\xi$. Now, as $P$ is torsionfree and $\bigoplus_S\bar S^{n}$ is divisible, $Q$ must be torsionfree divisible.

Let $Q=Q'\oplus Q''$ be a decomposition and $Q'\neq0$. Both summands are torsionfree divisible, so the composite $P\to Q\to Q'$ is a mono with regular cokernel $C$ by Lemma~\ref{le:sub-P}. We obtain an exact sequence $0 \to Q'' \to \bigoplus\nolimits_S\bar S^n \to C\to 0$ which splits by Lemma~\ref{le:ext-divisible-reg}. As $Q''$ is torsionfree divisible, it must be zero. Hence $Q$ is indecomposable.
\end{proof}

For the second construction we recall that there is a lattice isomorphism $\bfL(X^I)\cong\bfL(X)$. As each Prüfer module $\bar S$ is $\Sigma$-pure injective, Lemma~\ref{le:prufer}, so too is $\bar S^I$ by Proposition~\ref{pr:subgroup-fin-def}, and hence the pure epimorphism $\bar S^I\twoheadrightarrow\frf(\bar S^I)$ splits by Remark~\ref{re:quot-lattice}.

\begin{lem}\label{le:generic-construction2} 
The module $\frf(\bar S^I)$ is torsionfree divisible, and it is zero if and only if $I$ is finite.
\end{lem}

\begin{proof}
For $R\in\R$ we have $\Ext^1_\Lambda(R,\bar S^I)\cong\Ext^1_\Lambda(R,\bar S)^I=0$, so $\bar S^I$ is divisible. It follows that $\frf(\bar S^I)$ is torsionfree divisible. If $I$ is infinite, then there is an inclusion $\bar S^{\bbN}\subseteq\bar S^I$, and any element $(x_n)$ with $x_n\in S_n-S_{n-1}$ lies in $\bar{S}^I$, but not in its torsion part $\frt(\bar{S}^I)$. Therefore $\frf(\bar{S}^I)$ is not zero.
\end{proof}

\subsection*{The category of torsionfree divisible modules}

The category of torsionfree divisible modules consists of semisimple objects, and the generic module is the unique simple object in this category.

\begin{thm}
There is up to isomorphism a unique generic $\Lambda$-module $Q$. Its endomorphism ring $\Gamma=\End_\Lambda(Q)$ is a division ring and the assignment $X\mapsto X\otimes_\Gamma Q$ yields an equivalence
\[ \Mod\Gamma \iso \R^\perp. \]
\end{thm}

\begin{proof}
From Lemma~\ref{le:generic-construction} or Lemma~\ref{le:generic-construction2} we know that there exists an indecomposable torsionfree divisible $\Lambda$-module $Q$. Up to isomorphism this is the unique generic module, by Theorems~\ref{th:endofinite} and \ref{th:endofinite2}, and its endomorphism ring $\Gamma$ is a division ring by Lemma~\ref{le:end-generic}. Now every $X\in\R^\perp$ is a direct sum of copies of $Q$. For, $X$ is endofinite by Theorem~\ref{th:endofinite}, so admits a decomposition into indecomposables by Corollary~\ref{co:decomposition}, and each of these must be generic, so isomorphic to $Q$. It follows that $-\otimes_\Gamma Q$ is an equivalence, with quasi-inverse $\Hom_\Lambda(Q,-)$.
\end{proof}

There is a more general approach to the existence of $Q$. In fact,
every infinite product of copies of a Pr\"ufer module contains $Q$ as
a direct summand. This can be deduced from a general result: if
$M$ is an indecomposable module which admits a surjective, locally
nilpotent endomorphism with kernel of finite length, then the
countable product $M^{\mathbb N}$ has a direct summand which is a
generic module \cite{K1,endofin}.  A construction and examples of
modules $M$ with these properties are given in \cite{ladder}.

A further point of view is provided by work of Cohn and Schofield
\cite{Scho}. There is a ring epimorphism $\lambda\colon\La\to\La_\R$,
called the \emph{universal localisation} of $\La$ at $\R$, which
inverts the projective resolutions of the simple regular modules and
is universal with this property. It is easy to see that $\Mod\La_\R$
is equivalent to the category $\R^\perp$ of torsionfree divisible
modules. Moreover, it is shown in \cite{CB} that $\La_\R$ is a simple
artinian ring whose unique simple module is isomorphic to $Q$. Note
that $\Tor_1^\La(\La_\R,\La_\R)=0$, so the ring epimorphism $\lambda$
is \emph{pseudo-flat}. Since $\La$ is hereditary also the higher
Tor-groups vanish, that is, $\lambda$ is a \emph{homological ring
  epimorphism}.

In fact, this is a special instance of a more general
phenomenon. Recall that a module is called a \emph{brick} if its
endomorphism ring is a division ring, and it is called a \emph{stone}
if it is an endofinite brick without self extensions.

\begin{prop}[\cite{R4}]
  For any finite dimensional algebra $A$, there is a
  one-to-one-correspondence between isomorphism classes of endofinite
  bricks over $A$ and equivalence classes of ring epimorphisms from
  $A$ to simple artinian rings, which restricts to a bijection between
  stones and pseudoflat ring epimorphisms.
\end{prop}
Here two ring epimorphisms $\lambda_i\colon A\to B_i$, $i=1,2$, are
said to be \emph{equivalent} if there is a ring isomorphism $\mu$ such
that $\lambda_1=\mu\lambda_2$.
\begin{proof}
  To every endofinite brick $M$ over $A$ one can associate a ring
  epimorphism $\lambda\colon A\to B$  by
  taking $B\coloneqq\End_D(M)$ where $D=\End_A(M)$, and by defining
  $\lambda(a)$ to be right multiplication by the element $a$ on
  $M$.

  Conversely, if $\lambda\colon A\to B$ is a ring epimorphism and $B$
  simple artinian, then for some $d\ge 1$, $B$ is the full $d\times d$ matrix ring over
  some division ring $D$, and there
  is a unique simple $B$-module $M$ up to isomorphism, which is
  isomorphic to $D^d$ as a $D$-module. By regarding $M$ as an
  $A$-module via $\lambda$ we obtain an endofinite brick in $\Mod{A}$.

  This yields the stated one-to-one-correspondence.  That stones
  correspond to pseudo-flat ring epimorphisms under this bijection
  follows from \cite[Theorem 4.8]{Scho}, because $M$ is a stone if and
  only if the image of $\Mod B$ in $\Mod A$ via restriction of
  scalars is closed under extensions.
\end{proof}

\begin{rem}
  (1) For hereditary algebras there is a further
  one-to-one correspondence between stones and indivisible Schur
  roots. We refer to \cite[3.3]{CB2} for details.

  (2) There are many large bricks over $\La$. An infinite semibrick
  (i.e.~a collection of Hom-orthogonal bricks) consisting of infinite
  length submodules of the generic module is constructed in
  \cite[6.9]{Ri1979}. This is used to embed the module category of a
  wild hereditary algebra as a full exact abelian subcategory into
  $\Mod\La$. When $\La$ is the Kronecker algebra a more direct
  construction of such a semibrick is provided in \cite{tamewild}.  It
  leads to a full exact embedding into $\Mod\La$ of the (wild)
  category of representations of the $3$-Kronecker quiver.
\end{rem}

\section{Global results}

The results in the previous sections yield some insight into the
global structure of the category of $\La$-modules. In this section we
make this more explicit, using the notion of a cotorsion pair. In particular, we
discuss which properties are characteristic for $\La$ to be of tame
representation type.  In fact, some of the results in this section
hold true more generally for canonical algebras, as shown in
\cite{RR}. Then we complete the classification of the pure-injective $\La$-modules.
This requires a refinement of the $k$-duality between left and right
$\Lambda$-modules which is known as
elementary duality. 

We start out with an immediate consequence of
Proposition~\ref{pr:divisible-pure-inj}.

\begin{prop}\label{pr:preinjectives-split}
For every $\La$-module $M$ the inclusion $\fri(M)\to M$ splits.
\end{prop}
\begin{proof}
  The canonical sequence $0\to \fri(M)\to M\to M/\fri(M)\to 0$ is
  pure exact. The preinjective module $\fri(M)$ is clearly divisible
  and therefore pure-injective. Thus the sequence is split exact.
\end{proof}

\subsection*{A trisection}

Recall that two classes of modules $\X,\Y \subseteq \Mod\La$
form a \emph{cotorsion pair} $(\X,\Y)$ if
\[\X = \{X\in\Mod\La\mid\Ext^1_\La(X,Y)=0\text{ for all
  }Y\in\Y\}\]
and
\[\Y = \{Y\in\Mod\La\mid\Ext^1_\La(X,Y)=0\text{ for all }X\in\X\}.\]
The pair is \emph{complete} if for all $M\in\Mod\La$ there exist short exact
sequences
\[0\to M \xrightarrow{f} Y^M \to X^M\to 0\qquad\text{and}\qquad
0\to Y_M\to X_M\xrightarrow{g} M\to 0\]
such that $X_M,X^M\in\X$ and $Y_M,Y^M\in \Y$. In \cite{ET} it is proven that every set of
modules $\X_0$ \emph{generates} a complete cotorsion pair
$(\X,\Y)$ where
\[\Y = \{Y\in\Mod\La\mid\Ext^1_\La(X,Y)=0\text{ for all }X\in\X_0\}.\]
Moreover, if $\X$ is closed under filtered colimits, then one knows
from \cite{E} that $f$ and $g$ can be chosen such that $f$ is  left
minimal and $g$ is right
minimal. In other words, $f$ is a $\Y$-envelope, $g$ is a $\X$-cover,
and these approximation sequences are uniquely determined up to
isomorphism.

Let us consider the following subcategories.
\begin{align*}
  \C&\coloneqq\{X\in\Mod\La\mid\Hom_\Lambda(\I,X)=0\}\\
  \D&\coloneqq\{X\in\Mod\La\mid\Hom_\Lambda(X,\R)=0\}\\
  \E&\coloneqq\{X\in\Mod\La\mid\Hom_\Lambda(\D,X)=0\}\\
  \omega&\coloneqq\C\cap\D
\end{align*}
Then we have $\C =\overline{\P\vee\R}$ and the objects in $\D$ are
precisely the divisible modules. The modules in $\E$ are called
\emph{reduced}. The notation reflects the fact that the category of
divisible modules fits into a cotorsion pair  $(\C,\D)$ and a torsion
pair  $(\D,\E)$, as we now explain.

\begin{prop} 
  The pair $(\C,\D)$ is a complete cotorsion pair that is generated by the
  class $\R$ of finite dimensional regular modules. Moreover, there
  are the following approximation sequences.
\begin{enumerate}
\item For each $M\in\C$ there exists a short exact sequence
\[ 0\to M\xrightarrow{f^\omega} D^M\to C^M\to 0\]  such that
$f^\omega$ is a $\D$-envelope and $D^M,C^M\in \omega$.
\item For each $M\in\D$ there exists a short exact sequence
\[ 0\to D_M \to C_M\xrightarrow{g_\omega} M\to 0\] such that
$g_\omega$
is a $\C$-cover and $C_M,D_M\in\omega$.
\end{enumerate}
\end{prop}
\begin{proof}
  By definition $\X_0=\R$ generates a complete cotorsion pair $(\X,\Y)$
  with $\Y=\D$. We need to show that $\X=\C$. The inclusion
  $\X\subseteq\C$ follows from the Auslander--Reiten formula,
  and the reverse inclusion is Corollary~\ref{co:ext-divisible}.

The existence of the minimal approximations in (1) and (2) is now a consequence of the fact that
 $\C$ is closed under filtered colimits.
\end{proof}

\begin{cor}
  The pair $(\D,\E)$ is a split torsion pair.
\end{cor}
\begin{proof}
It is clear that  $\R\subseteq\E\subseteq\C$. Thus $(\D,\E)$ is a
torsion pair which splits as $(\C,\D)$ is a cotorsion pair.
\end{proof}

We obtain the following structural result for $\Mod\La$.

\begin{thm}\label{th:decompos}
  Every $\La$-module can be decomposed as
  \[M=E\oplus Q^{(\alpha)}\oplus R\oplus I\] where $E$ is reduced, $R$
  a direct sum of Pr\"{u}fer modules, $I$ a direct sum of
  indecomposable preinjective modules, and $\alpha$ a cardinal.

Moreover, the category $\Mod\La$ admits the following trisection:
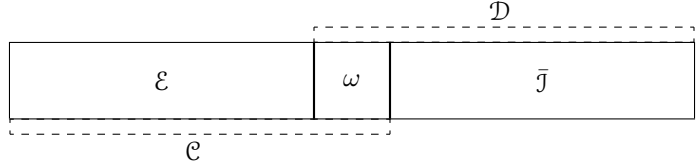
\begin{center}
\setlength{\unitlength}{1.0mm}
\begin{picture}(90,18)
\put(0,0){\framebox(40,10){$\E$}}
\put(40,0){\framebox(10,10){$\omega$}}
\put(50,0){\framebox(40,10){$\bar\I$}}
\put(0,-2.1){\dashbox(50,2){}}
\put(23,-5.5){$\C$}
\put(40,10.1){\dashbox(50,2){}}
\put(63,13){$\D$}
\end{picture}
\end{center}
\vspace*{2em} The direction of non-zero morphisms is from left to
right and any morphism from $\E$ to $\bar\I$ factors through
$\omega$.
\end{thm}

\begin{proof}
  Using that $(\D,\E)$ is a split torsion pair we decompose
  $M=M'\oplus E$ with $M'\in\D$ and $E\in\E$. According to
  Proposition~\ref{pr:preinjectives-split}, we can further write
  $M'=I\oplus W$ with $I\in\bar\I$ and $W\in\omega$.  Next consider
  the pure exact sequence $0\to\frt(W)\to W\to \frf(W)\to 0$, which
  splits since divisible modules are $\Sigma$-pure injective by
  Proposition~\ref{pr:divisible-pure-inj}.  Therefore
  $\frt(W)\in\bar{\R}$ is divisible and $\frf(W)$ is torsionfree
  divisible.  It follows from Lemma~\ref{le:ext-divisible-reg} and
  Theorem~\ref{th:endofinite} that $\frt(W)$ is a direct sum of
  Pr\"ufer modules, and $\frf(W)$ is a direct sum of copies of the
  generic module $Q$. Thus $W$ is of the required form.

  For the trisection, observe that every reduced module admits a
  $\D$-envelope lying in $\omega$. This yields the stated
  factorisation property.
\end{proof}

Notice that the splitting result shown above under the assumption
that $\La$ has tame representation type is in fact characteristic for
the tame case, as is the existence of a torsionfree generic
module. Thus let us assume for a moment that $\La$ is a connected hereditary
algebra of infinite representation type. Note that we still have the decomposition
$\mod\La=\P\vee\R\vee\I$ as in the tame case.

\begin{thm} 
The following statements are equivalent.
\begin{enumerate}
\item The algebra $\La$ is of tame representation type.
\item The torsion pair $(\bar\I,\overline{\P\vee\R})$ is a split torsion pair.
\item Every divisible module is pure-injective. 
\item The pair $(\C, \D)$ is a cotorsion pair.
\item There is a torsionfree generic $\La$-module.
\end{enumerate}
\end{thm}
\begin{proof}
  The equivalence of the first three conditions is proven in \cite[3.7
  and 3.9]{Ri1979} and \cite[Theorem 4.6]{L2}. (1)$\Leftrightarrow$(4)
  is \cite[Theorem 18]{AKT}, and (1)$\Leftrightarrow$(5) is shown in
  \cite{genrep}.\end{proof}

\subsection*{Elementary duality}

We discuss a refined version of the duality $D=\Hom_k(-,k)$ between right and left $\Lambda$-modules that is known as \emph{elementary duality}. It can be defined for an arbitrary $k$-algebra. Here it will be used to establish the classification of all pure-injective $\Lambda$-modules.

Let $\A\coloneqq\Add(\mod\Lambda^\op,\Ab)$ denote the category of additive functors $\mod\Lambda^\op\to\Ab$. The assignment
\[ X \longmapsto T_X\coloneqq X\otimes_\Lambda- \]
provides a useful embedding $\Mod\Lambda\to\A$. In fact, $X$ is pure-injective if and only if the object $T_X$ is injective, and every injective object $E$ in this functor category is isomorphic to $T_Y$ for $Y=E(\Lambda)$. Moreover, for every coherent $F$ there is a natural isomorphism
\[ F(X) \cong \Hom(F^\vee,T_X). \]

A pure-injective $\Lambda$-module $M$ is called \emph{simply reflexive} if there is a coherent functor $F$ such that $F(M)\neq 0$ but $F'(M)=0$ or $(F/F')(M)=0$ for every coherent subfunctor $F'\subseteq F$. In this case we call $(M,F)$ a \emph{simple pair}.

\begin{lem}\label{le:simply-reflexive}
Every simply reflexive pure-injective module has an indecomposable direct summand. A non-zero pure-injective module satisfying acc or dcc on subgroups of finite definition is simply reflexive.
\end{lem}

\begin{proof}
Let $(M,F)$ be a simple pair. Choose a non-zero element in $F(M)$, corresponding to a morphism $\phi\colon F^\vee\to T_M$. An injective envelope $\Im\phi\rightarrowtail T_N$ in the functor category $\A$ provides an indecomposable direct summand $N$ of $M$. This follows since $F^\vee$ is a simple object in the localisation of $\A$ with respect to $^\perp{T_M}$, and an injective envelope of a simple object is indecomposable; see \cite[Chap.~III, Proposition~6]{Ga1962}.\footnote{
Let $\A/\C$ be the localisation with respect to a Serre subcategory $\C\subseteq\A$. An object $X\in\A$ is simple in $\A/\C$ if and only if $X\not\in\C$ and for every subobject $X'\subseteq X$ we have $X'\in\C$ or $X/X'\in\C$.
}
  
Now suppose that $M$ is a pure-injective module satisfying dcc on subgroups of finite definition. Choose a coherent subfunctor $F\subseteq \Hom_\Lambda(\Lambda,-)$ such that $F(M)\neq 0$, and minimal with respect to this property. Then $(M,F)$ is a simple pair.  If instead $M$ satisfies acc, then $DM$ satisfies dcc and there is a simple pair $(DM,F)$. Thus $(M,F^\vee)$ is a simple pair.
\end{proof}

As in Section~\ref{se:purity} let $\Coh(\Lambda)$ denote the abelian category of coherent functors on $\Mod \Lambda$, and write $\bfL$ for the modular lattice of coherent subfunctors of the forgetful functor. Each $\Lambda$-module $M$ determines a Serre subcategory
\[ \C_M \coloneqq \{ F \in \Coh(\Lambda) \mid F(M)=0 \} \]
and the lattice $\bfL(M)$ given by evaluation at $M$ identifies with the lattice of subfunctors of the forgetful functor in the Serre quotient $\Coh(\Lambda)/\C_M$.

\begin{prop}[\cite{He1993,Kr1997}]\label{pr:duality}
Let $M$ be a simply reflexive indecomposable pure-injective $\Lambda$-module. Then there exists a simply reflexive indecomposable pure-in\-jec\-tive $\Lambda^\op$-module $M^\vee$ with the following property: for each coherent functor $F$ on $\Mod\Lambda^\op$ we have
\[ F(M^\vee) = 0 \quad \iff \quad F^\vee(M)=0. \]
The module $M^\vee$ is unique up to isomorphism and a direct summand of $DM$. Moreover, $M\cong (M^\vee)^\vee$.
\end{prop}

\begin{proof}
Choose $F$ coherent such that $(M,F)$ is a simple pair. Then $(DM,F^\vee)$ is a simple pair, and Lemma~\ref{le:simply-reflexive} yields the desired simply reflexive indecomposable summand $M^\vee$ of $DM$. Note that $\C_{M^\vee}\supseteq \C_{DM}=(\C_M)^\vee$ as $M^\vee$ is a direct summand of $DM$. Now let $G\in\Coh(\Lambda^\op)$. Then $G\not\in (\C_M)^\vee$ implies $\Hom(G,T_M)\neq 0$. Thus $F^\vee$ is isomorphic to a subquotient of $G$ in $\Coh(\Lambda^\op)/(\C_M)^\vee$, and the same holds in $\Coh(\Lambda^\op)/\C_{M^\vee}$. The construction of $M^\vee$ implies that $F^\vee\not\in\C_{M^\vee}$ since $\Hom(F,T_{M^\vee})\neq 0$. Thus $G\not\in \C_{M^\vee}$ and therefore $\C_{M^\vee}=(\C_M)^\vee$. The uniqueness of $M^\vee$ follows from the uniqueness of an injective envelope, and the same reason yields the isomorphism $M\cong (M^\vee)^\vee$.
\end{proof}  

Our techniques yield another result which we used earlier in Corollary~\ref{co:endofinite} to describe the definable subcategory generated by an endofinite module.

\begin{lem}\label{le:simply-reflexive-endofinite}
Let $M$ be an endofinite module. Then any indecomposable module belonging to the definable subcategory generated by $M$ is isomorphic to a direct summand of $M$.
\end{lem}

\begin{proof}
The lattice of subobjects of the forgetful functor in $\Coh(\Lambda)/\C_M$ identifies with $\bfL(M)$, which has finite height by Proposition~\ref{pr:subgroup-fin-def}. It follows that every object in $\Coh(\Lambda)/\C_M$ has finite length, and there is a finite representative set of simple objects $F_1,\ldots F_r$.

Now let $N$ be an indecomposable module satisfying $\C_M\subseteq\C_N$. The category $\Coh(\Lambda)/\C_N$ is a non-zero quotient of $\Coh(\Lambda)/\C_M$, so there is at least one index $i$ such that $F_i(N)\neq0$. We now argue as in the proof of Lemma~\ref{le:simply-reflexive}. Choose a non-zero element in $F_i(M)$, corresponding to a morphism $\phi\colon F_i^\vee\to T_M$. An injective envelope $\Im\phi\rightarrowtail T_{M'}$ in $\A$ provides an indecomposable direct summand $M'\subseteq M$ since $F_i^\vee$ is a simple object in the localisation of $\A$ with respect to $^\perp{T_M}$. On the other hand, $\Hom(F_i^\vee,T_N)\neq 0$. Thus $T_N$ is an injective envelope as well, and therefore $N\cong M'$.
\end{proof}

\subsection*{The classification of all pure-injectives}

We establish the classification of all pure-injective
$\La$-modules. First of all, we need to introduce the dual counterpart
of the Pr\"ufer modules. For every simple regular module $S\in\R$ we
denote by $\hat{S}$ the \emph{adic module} constructed as the inverse
limit of the chain of irreducible epimorphisms ending in $S$.  Note
that $\hat S=D\bar T$ for $T=DS$ since
$D(\colim_{n\ge 1}T_n)=\lim_{n\ge 1}DT_n$.  This shows in particular
that $\hat{S}$ is torsionfree, reduced, and regular; it is
indecomposable as shown e.g.~in \cite[p.~49]{CB2}.

The following describes the elementary duality between modules over
$\La$ and $\La^\op$.

\begin{prop}
  Let $M$ be an indecomposable $\La$-module of finite length. Then $M^\vee=DM$. For
  $S\in\R$ simple we have $\bar S^\vee=D\bar S=\hat T$ for
  $T=S^\vee$, and therefore  $\hat S^\vee=\bar T$. Finally,
  $(Q_\La)^\vee=Q_{\La^\op}$.\qed
\end{prop}

We are ready to provide the complete description of
the pure-injective $\La$-modules.

\begin{thm}\label{th:classification}
The following is, up to isomorphism, a complete list of indecomposable
pure-injective $\La$-modules.
\begin{enumerate}
\item The indecomposable modules of finite length (preprojective,
  regular, and preinjective).
\item For each simple regular $S$ the Pr\"ufer module $\bar S$
  and the adic module $\hat S$.
\item The unique torsionfree divisible module $Q$. 
\end{enumerate}
Moreover, each pure-injective $\La$-module $X$ is a pure-injective
envelope of a direct sum $\bigoplus_{i\in I}X_i$ of indecomposable
pure-injective modules, and the $X_i$ are uniquely determined, up to
isomorphism, by $X$.  
\end{thm}

\begin{proof}
  For the list of indecomposable pure-injectives, see
  \cite{DR1976,Ok1980,Pr1988}, but implicitly also \cite{Ge1985}. We
  argue as follows. Let $M$ be indecomposable and
  pure-injective. Suppose first that $\frt(M)\neq 0$. In this case $M$
  is torsion and belongs to the list, because $\frt(M)$ is a pure submodule of $M$ and
  contains an indecomposable pure-injective summand, which is either
  in $\R\vee\I$ or a Pr\"ufer module; see
  Lemma~\ref{le:torsion-indec}. Now suppose that $M$ is
  torsionfree. Then its dual $DM$ is divisible and therefore
  $\Si$-pure injective, so satisfies dcc on subgroups of finite
  definition by Proposition~\ref{pr:divisible-pure-inj}. Thus $M$
  satisfies acc on subgroups of finite definition and is therefore
  simply reflexive by Lemma~\ref{le:simply-reflexive}. Now $M^\vee$ is
  a direct summand
  of $DM$ by 
  Proposition~\ref{pr:duality}, so is either preinjective, a Pr\"ufer module, or
  generic, by Theorem~\ref{th:decompos}. It follows that $M$ is either
  preprojective, an adic module, or generic. Thus the above list of
  indecomposable pure-injective $\La$-modules is complete.

  For the second statement we claim that every non-zero pure-injective
  $\La$-module $M$ admits an indecomposable direct summand. We argue
  as before. The claim is clear when $\frt(M)\neq 0$. If $M$ is
  torsionfree, then it satisfies acc on subgroups of finite definition
  and has therefore an indecomposable direct summand by
  Lemma~\ref{le:simply-reflexive}.  This yields the claim, and then
  one applies \cite[Theorem~8.28]{JL1989} which reduces the assertion
  to a general fact about injective objects in Grothendieck categories
  due to Gabriel and Oberst \cite{GO1966}.
 
  A different strategy for proving the theorem follows from
  \cite[Theorem~8.53]{JL1989}. Their result provides a complete
  description of the pure-injective modules over a ring $A$ when the
  category $\Add(\mod A^\op,\Ab)$ has Krull dimension. For  a
  tame hereditary algebra $A$ this category has Krull dimension $2$ by
  \cite{Ge1985}, and the explicit description of the Krull filtration
  provides the list of the indecomposable pure-injectives.
\end{proof}

\begin{rem}
  The pure-injective divisible modules are completely described by
  Theorem~\ref{th:decompos}. The reduced regular ones are investigated
  in {\cite[2.2, 2.3, 2.4]{R1}}. They are precisely the modules in
  $\Prod\R$ and admit a description that is somewhat dual to
  Proposition~\ref{pr:tor-reg}. Here, $\Prod\R$  denotes the smallest full
subcategory of all $\La$-modules that contains $\R$ and is closed under all products and
direct summands. Then  every reduced regular module has a
  unique decomposition $M=\prod_{x\in\mathbb X}M_x$ where
  $M_x\in\Prod{\R_x}$, and there are pure exact sequences
  $0\to L_x\to M_x\to N_x\to 0$ where $L_x$ is a direct sum of modules
  in ${\R_x}$ and adic modules obtained from ${\R_x}$, and $N_x$ is a
  divisible module.\qed
\end{rem}

\begin{rem}
  For any ring the set of isomorphism classes of indecomposable
  pure-injective modules carries a topology, which Ziegler introduced
  in model-theoretic terms \cite{Zi1984}. The closed subsets of this
  \emph{Ziegler spectrum} are precisely the sets of indecomposables
  lying in a definable subcategory \cite{CB2}.  A complete description of this
  spectrum is known only in few cases, including rings of finite
  representation type (when it is discrete) or Dedekind domains. For a
  tame hereditary algebra the Ziegler spectrum has been determined
  independently by Prest \cite{Pr1998} and Ringel \cite{Ri1998}.
\end{rem}

The following diagram shows how the indecomposable pure-injective
$\La$-modules are distributed among the various subcategories of
$\Mod\La$. 
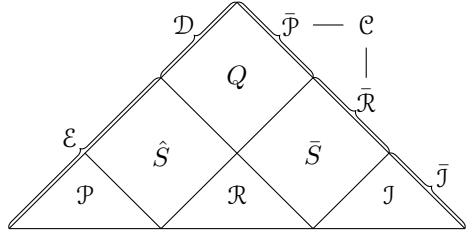
\medskip

\begin{center}
    \begin{tikzpicture}
        \draw[black,-] (-3,0) -- (3,0);
        \draw[black,-] (-3,0) -- (0,3);
        \draw[black,-] (0,3) -- (3,0);
        \draw[black,-] (1,0) -- (2, 1);
        \draw[black,-] (-1,0) -- (1,2);
        \draw[black,-] (-1,0) -- (-2, 1);
        \draw[black,-] (1,0) -- (-1,2);
        \node[black] at (1,1) {$\bar S$};
        \node[black] at (-1,1) {$\hat S$};
        \node[black] at (-2,0.45) {$\mathcal{P}$};
        \node[black] at (0,0.45) {$\mathcal{R}$};
        \node[black] at (2,0.45) {$\mathcal{I}$};
        \node[black] at (0,2) {$Q$};
        \draw[black,decorate,decoration={brace}] (-3,0) -- (-1,2);
        \draw[black,decorate,decoration={brace}] (-1,2) -- (0,3);
        \draw[black,decorate,decoration={brace}] (0,3) -- (1,2);
        \draw[black,decorate,decoration={brace}] (1,2) -- (2,1);
        \draw[black,decorate,decoration={brace}] (2,1) -- (3,0);
        \node[black] at (-2.2,1.2) {$\mathcal{E}$};
        \node[black] at (-0.7,2.7) {$\mathcal{D}$};
        \node[black] at (0.7,2.7) {$\bar{\mathcal{P}}$};
        \node[black] at (1.7,1.7) {$\bar{\mathcal{R}}$};
        \node[black] at (2.7,0.7) {$\bar{\mathcal{I}}$};
        \node[black] at (1.7,2.7) {$\mathcal{C}$};
        \draw[black,-] (1.7,2) -- (1.7,2.4);
        \draw[black,-] (1,2.7) -- (1.4,2.7);
\end{tikzpicture}
\end{center}
The direction of non-zero morphisms is from left to right. The
horizontal layers of the diagram reflect the Krull filtration of the
functor category $\Add(\mod \La^\op,\Ab)$, which corresponds to the
\emph{Cantor--Bendixson filtration} of the Ziegler
spectrum. Specifically, the first layer consists of all
indecomposables of finite length, which are precisely the isolated points.  The
second layer consists of the Pr\"ufer and the adic modules, which are
the isolated points after removing the first layer.  The final layer is given by
the generic module.

\subsection*{Acknowledgement}
The authors are grateful to Maximilian Kaipel and an anonymous referee for
several helpful comments concerning the exposition of this material.
The work was supported by the Deutsche Forschungsgemeinschaft
(Project-ID 491392403 – TRR~358) and by the Program FIS2021 of the
Italian Ministry of University and Research (Project LAVIE, FIS
00001706).

\appendix

\section{Direct sums of \texorpdfstring{$\Sigma$}{Sigma}-pure injectives}

The article \cite{Ri1979} of Ringel was presented at a conference on abelian groups, and explains to what extent the well-developed theory of abelian groups extends to the wider class of modules over a finite dimensional tame hereditary algebra. The connections being clearest when one restricts to the category of regular modules.

Of particular importance when studying large abelian $p$-groups is the notion of a basic subgroup. This is a pure subgroup which is a direct sum of finite $p$-groups such that the cokernel is divisible \cite[Chapter 5.5]{Fu2015}. As Fuchs writes \cite[p.172]{Fu2015}
\begin{quote}
Kulikov\dots developed the theory of basic subgroups\,\dots. Szele should be given the credit for recognizing their utmost importance\,\dots. It is hard to think of the theory of $p$-groups without basic subgroups.
\end{quote}
In \cite[Theorem 1.G]{Ri1979} Ringel extended this idea to modules over a finite dimensional algebra, showing that every module admits a pure submodule which is a direct sum of finite dimensional modules and whose cokernel has no finite dimensional direct summand. In both cases one can thus reduce problems to `purely infinite' modules. We note that these ideas have been used on several occasions in the main text, notably in Propositions~\ref{pr:preinj} and \ref{pr:tor-reg}, and Corollary~\ref{co:endofinite}. 

In this appendix we revisit this construction in the context of a locally finitely presented Grothendieck category. In Proposition~\ref{pr:X-basic} we show that, given a family $\X$ of indecomposable $\Sigma$-pure injectives, every object contains a pure subobject which is  a coproduct of elements from $\X$ and satisfies a certain maximality condition. When $\add\X$ is the category of finite abelian $p$-groups, respectively finite dimensional modules, one recovers the situations described above, whereas the pure exact sequence from Proposition~\ref{pr:tor-reg} arises when we take $\add\X$ to be a tube for a tame hereditary algebra.

In Proposition~\ref{pr:preinj} we showed that filtered colimits of finite dimensional preinjective modules are direct sums of preinjective modules. In other words, taking $\add\X=\I$, we have $\bar\I=\Add\X$. We generalise these ideas in Theorem~\ref{th:loc-T-nilp}, together with Lemma~\ref{le:no-pure-ext}. Finally, Theorem~\ref{th:endofinite-definable} gives an alternative approach to Corollary~\ref{co:endofinite}.

\subsection*{Basic subobjects}

Let $\A$ be a locally finitely presented Grothendieck category, and write $\fp\A$ for the subcategory of finitely presented objects.

An exact sequence $0\to X\to Y\to Z\to 0$ in $\A$ is called \emph{pure} provided it is a filtered colimit of split exact sequences, equivalently the sequence $0\to\Hom(C,X)\to\Hom(C,Y)\to\Hom(C,Z)\to0$ is exact for all $C\in\fp\A$ (see for example \cite[\S12.1]{Kr2022}). It follows that pushouts and pullbacks of pure exact sequences are again pure exact, and pure monomorphisms, respectively pure epimorphisms, are closed under composition. In other words, the pure exact sequences form an exact structure on $\A$, and we have the associated sub-bifunctor $\PExt^1(-,-)$ of $\Ext^1(-,-)$. We call $X\in\A$ \emph{pure-injective} provided $\PExt^1(-,X)=0$, and \emph{$\Sigma$-pure injective} if $X^{(I)}$ is pure-injective for all sets $I$.

We fix a collection $\X$ of indecomposable $\Sigma$-pure injective objects in $\A$ and write $\add\X$, respectively $\Add\X$, for the smallest additive subcategory containing all finite, respectively arbitrary, coproducts of objects from $\X$.

We observe that every $M\in\add\X$ is a finite direct sum of objects in $\X$. For, every indecomposable pure-injective object has local endomorphism ring by Lemma~\ref{le:end-generic}, so the result follows from the Krull--Remak--Schmidt Theorem, \cite[Corollary 12.7]{AF1992}. The analogous result for $\Add\X$ fails in general, so we introduce the subcategory $\Coprod\X$ given by arbitrary coproducts of objects from $\X$. 

We write $\PExt^1(\X,\Add\X)=0$ provided every pure extension of an object in $\X$ by something in $\Add\X$ is split. This condition is clearly satisfied if $\X$ is finite, since then every object in $\Add\X$ is pure-injective, but is also satisfied if $\X\subseteq\fp\A$, since then every $X\in\X$ is pure-projective.

Given $M\in\A$ we call $M'\subseteq M$ an \emph{$\X$-basic subobject} provided it is a pure subobject with $M'\in\Coprod\X$, and there is no larger pure subobject $M'\oplus X$ with $X\in\X$.

\begin{prop}\label{pr:X-basic}
Every $M\in\A$ admits an $\X$-basic subobject. Suppose instead $N\in\Coprod\X$ is a pure subobject of $M$. If $M/N$ has no direct summand from $\X$, then $N$ is $\X$-basic. The converse holds whenever $\PExt^1(\X,\Add\X)=0$.
\end{prop}

\begin{proof}
We apply Zorn's Lemma as in the proof of Corollary~\ref{co:decomposition}. Let $\mathcal U$ denote the set of subobjects of $M$ which are isomorphic to an object from $\X$, and consider
\[ \mathcal S \coloneqq \{ \mathcal V \subseteq\mathcal U \mid \textstyle{\coprod_{L\in\mathcal V}L\to M} \textrm{ is a pure monomorphism} \}. \]
Now $\mathcal S$ is non-empty (as zero is a pure subobject of $M$), and is partially ordered via inclusion. Also, given a chain $\mathcal V_i$ in $\mathcal S$, let $\mathcal V=\bigcup_i\mathcal V_i$ be its union. Then $\coprod_{L\in\mathcal V}L\to M$ is again a pure monomorphism, so $\mathcal V\in\mathcal S$. Zorn's Lemma now yields a maximal element $\mathcal W\in\mathcal S$, in which case $M'=\coprod_{L\in\mathcal W}L$ is an $\X$-basic subobject of $M$.

Now let $N\subseteq M$ be a pure subobject with $N\in\Coprod\X$. If $N$ is not $\X$-basic, then we can find a larger pure subobject $N\oplus X$ for some $X\in\X$. The pushout along $N\oplus X\to X$ of the pure monomorphism $N\oplus X\to M$ yields a pure, and hence split, monomorphism $X\to M/N$.

Conversely, assume $N$ is $\X$-basic and let $X\to M/N$ be a pure monomorphism with $X\in\X$. The pullback lies in $\PExt^1(X,N)$, so if this vanishes, then $N\oplus X\subseteq M$. Moreover, as the composite $M\to M/N\to M/(N\oplus X)$ is a pure epimorphism, $N\oplus X\subseteq M$ is a pure subobject, a contradiction.
\end{proof}

\begin{rem}
If $\add\X$ is the category of finite abelian $p$-groups, then we recover the classical notion of a basic subgroup. If $\add\X=\mod\Lambda$ for a finite dimensional algebra $\Lambda$, then we obtain Ringel's construction from \cite[Theorem 1.G]{Ri1979}. 

If $\Lambda$ is a finite dimensional hereditary algebra and $\add\X=\I$ is the category of finite dimensional preinjective modules, then we obtain all but the uniqueness part of Proposition~\ref{pr:preinj}. If instead $\add\X$ is a tube, then we obtain the pure exact sequence from Proposition~\ref{pr:tor-reg}.
\end{rem}

In general there may be proper inclusions between $\X$-basic subobjects. Let $\add\X=\fp\A$ be a tube of rank one, for example the category of finite abelian $p$-groups, or finite length modules over a discrete valuation ring. Following \cite[Example 1.G]{Ri1979} let $M=\bigoplus_nS_n$, where $S_n$ is the unique indecomposable of length $n$. Let $\iota\in\End(M)$ be the sum of the inclusions $\iota_n\colon S_n\to S_{n+1}$. Then $1+\iota$ is a pure monomorphism with cokernel the Pr\"ufer module $\bar S$.

\begin{lem}\label{le:no-pure-ext}
Assume $\PExt^1(\X,\Add\X)=0$. If $\Add\X$ is closed under filtered colimits, then $\Add\X=\Coprod\X$ and there are no proper inclusions between $\X$-basic subobjects.
\end{lem}

\begin{proof}
Let $M'$ be an $\X$-basic subobject of some $M\in\Coprod\X$. Then $M/M'$ is a filtered colimit of direct summands of $M$, so our assumption gives $M/M'\in\Add\X$. If $M/M'$ is non-zero, then by Azumaya's Theorem (see \cite[Theorem 12.6]{AF1992}) it contains a direct summand from $\X$, in contradiction to Proposition~\ref{pr:X-basic}. We deduce that there are no proper inclusions between $\X$-basic subobjects.

It follows that $\Add\X=\Coprod\X$. For, suppose $M\oplus N\in\Coprod\X$ and let $M'\subseteq M$ and $N'\subseteq N$ be $\X$-basic subobjects. If $X\in\X$ is a direct summand of $(M/M')\oplus(N/N')$, then it must be a direct summand of either $M/M'$ or $N/N'$ as $\End(X)$ is local. This cannot happen, so $M'\oplus N'$ is an $\X$-basic subobject of $M\oplus N$ by Proposition~\ref{pr:X-basic}. Then $M'=M$ and $N'=N$ by the first part of the proof.
\end{proof}

\subsection*{Local T-nilpotence}

A morphism $f\colon X\to Y$ is \emph{radical} if for all $g\colon Y\to X$ the composite $gf$ lies in the Jacobson radical of $\End(X)$. If $X$ and $Y$ have local endomorphism rings, then $f$ is radical if and only if it is a non-isomorphism. By Lemma~\ref{le:end-generic} the latter holds if $X$ and $Y$ are both $\Sigma$-pure injective.

The family $\X$ is said to be \emph{locally T-nilpotent} if, given a sequence of radical morphisms $f_t\colon X_t\to X_{t+1}$ for $t\geq1$ between objects in $\X$, for each morphism $g\colon C\to X_1$ with $C\in\fp\A$ there exists an $n$ such that $f_n\cdots f_1g=0$. For example, this holds for the family of indecomposable preinjective modules over a finite dimensional hereditary algebra.

We observe that every finite family $\X$ is locally T-nilpotent. For, every $M\in\Add\X$ is $\Sigma$-pure injective, so pure subobjects are direct summands by Remark~\ref{re:quot-lattice}, and we can apply \cite[Proposition 4.2]{AH2003}. Also, every locally T-nilpotent family $\X$ satisfies $\Add\X=\Coprod\X$. For, local direct summands of $M\in\Add\X$ are direct summands by \cite[Proposition 4.2]{AH2003}, so \cite[Corollary 2.3]{GG} implies that the decomposition of every $N\in\Coprod\X$ complements direct summands.

\begin{lem}\label{lem:loc-T-nilp}
Let $\X$ be locally T-nilpotent, and let $f_t\colon Y_t\to Y_{t+1}$ for $t\geq1$ be a sequence of radical morphisms between objects $Y_t\in\add\X$. Then for each morphism $g\colon C\to Y_1$ with $C\in\fp\A$ there exists an $n$ such that $f_n\cdots f_1g=0$.
\end{lem}

\begin{proof}
We fix (finite) indecomposable decompositions $Y_t=\bigoplus_uY_{t,u}$ and thus regard each morphism as a matrix $f_t=(f_{t,uv})$. We next form a tree having root $C$ and vertices $(Y_{1,u_1},\ldots,Y_{t,u_t})$ such that the induced morphism $C\to Y_{1,u_1}\to\cdots\to Y_{t,u_t}$ is non-zero. Clearly this tree is connected and locally finite. If $f_n\cdots f_1g$ is non-zero for all $n$, then this tree must also be infinite, so by König's Lemma it contains an infinite path, contradicting the fact that $\X$ is locally T-nilpotent.
\end{proof}

\begin{thm}\label{th:loc-T-nilp}
Assume $\PExt^1(\X,\Add\X)=0$. If $\X$ is locally T-nilpotent, then $\Add\X$ is closed under filtered colimits. The converse holds if $\X\subseteq\fp\A$.
\end{thm}

We observe that if $\X$ consists of finitely presented $R$-modules for some ring $R$, then this theorem is covered by \cite[Theorem 4.4 and Proposition 4.2]{AH2003}. Also, taking $\add\X=\I$ to be the category of preinjective modules over an hereditary algebra as in Proposition~\ref{pr:preinj}, we see that $\bar\I=\Coprod\X$, so every filtered colimit of preinjective modules is a direct sum of preinjective modules.

\begin{proof}
Assume first that $\Add\X$ is closed under filtered colimits, and that $\X\subseteq\fp\A$. Take an infinite chain of non-isomorphisms $f_t\colon X_t\to X_{t+1}$ between objects in $\X$. Let $M$ be the colimit, so with morphisms $g_t\colon X_t\to M$ satisfying $g_t=g_{t+1}f_t$. By assumption $M\in\Add\X$, which equals $\Coprod\X$ by Lemma~\ref{le:no-pure-ext}, so if $M\neq0$, then there is a split monomorphism $p\colon X\to M$ for some $X\in\X$. As $X$ is finitely presented we can write $p=g_tp_t=g_{t+1}f_tp_t$ for some $p_t\colon X\to X_t$. Both $p_t$ and $f_tp_t$ are necessarily split monomorphisms between indecomposables, so isomorphisms, but $f_t$ is a non-isomorphism, a contradiction. Thus $M=0$, and hence $\X$ is locally T-nilpotent.

Suppose now that $\X$ is locally T-nilpotent and write $\mathcal C$ for the closure of $\add\X$ under filtered colimits. Take $M\in\mathcal C$ and $M'\subseteq M$ an $\X$-basic subobject. Then $M'$, and hence also $Y\coloneqq M/M'$ are both filtered colimits of direct summands of $M$. It follows that $Y\in\mathcal C$, but now with the added property that it has no direct summand from $\X$, by Proposition~\ref{pr:X-basic}. We need to show that $Y=0$.

Write $Y=\colim Y_\alpha$ with each $Y_\alpha\in\add\X$. We first claim that, for each $\alpha$, there exists $\beta\geq\alpha$ such that the morphism $Y_\alpha\to Y_\beta$ is radical. If not, then there exists some $\alpha$ such that, for each $\beta\geq\alpha$, the morphism $Y_\alpha\to Y_\beta$ is not radical. As $Y_\alpha\in\add\X$ is a finite sum, there must some indecomposable direct summand $Y'_\alpha$ such that each induced morphism $Y'_\alpha\to Y_\beta$ is not radical. As $\End(Y'_\alpha)$ is local, each sum morphism must be a split monomorphism, so the map $Y'_\alpha\to Y$ is a pure monomorphism, and hence split as $Y'_\alpha\in\X$ is pure-injective. Contradiction.

Now take $C\in\fp\A$ and a morphism $h\colon C\to Y$, say represented by $h_\alpha\colon C\to Y_\alpha$. By the preceding claim we can find a sequence of radical morphisms $Y_\alpha\to Y_\beta\to Y_\gamma\to\cdots$, so Lemma~\ref{lem:loc-T-nilp} applies to give $h=0$. We conclude that $Y=0$ as required.
\end{proof}

\subsection*{Endofinite objects}

An object $X\in\A$ is called \emph{endofinite} provided for each $C\in\fp\A$ the $\End(X)$-module $\Hom(C,X)$ has finite length.  In this case we can define the \emph{character}
\[ \character_X \colon \fp\A \to \bbN, \qquad C \mapsto \ell_{\End(X)}\Hom(C,X). \]
We observe that if $X=U\oplus V$, then $\character_U\leq\character_X\leq\character_U+\character_V$. Also, the characters of $X$, $X^{(I)}$ and $X^I$ agree for all sets $I$. For details on characters we refer to \cite{CB3}.

We note that every endofinite object is necessarily $\Sigma$-pure injective. In fact, the main results of Section~\ref{se:purity} concerning subgroups of finite definition extend to the category $\A$. Explicitly, we call a functor $F\colon\A\to\Ab$ \emph{coherent} if it is the cokernel of $\phi^\ast\colon\Hom(C',-)\to\Hom(C,-)$ for some morphism $\phi\colon C\to C'$ in $\fp\A$. As in Lemma~\ref{le:coherent} the coherent functors form an abelian category, so the coherent subfunctors of $\Hom(C,-)$ form a modular lattice $\bfL_C$. We write $\bfL_C(X)$ for the evaluation at $X$, whose elements $F(X)$ are called \emph{subgroups of finite definition} of $\Hom(C,X)$.

Now Proposition~\ref{pr:subgroup-fin-def} shows that $X$ is $\Sigma$-pure injective if and only if it has the descending chain condition (dcc) on subgroups of finite definition, and $X$ \emph{endofinite} if and only if has both dcc and acc on subgroups of finite definition, in which case the lattice $\bfL_C(X)$ identifies with the lattice of $\End(X)$-submodules of $\Hom(C,X)$. Similarly, by Remark~\ref{re:quot-lattice}, pure subobjects of $\Sigma$-pure injectives are themselves $\Sigma$-pure injective.

We finish by applying the theory of $\X$-basic objects to obtain Corollary~\ref{co:endofinite}.

\begin{thm}\label{th:endofinite-definable}
Let $X$ be endofinite, and take $\X$ to be a representative set of indecomposable direct summands of $X$. Then $\X$ is locally T-nilpotent, and $\Add\X=\Coprod\X$ is a definable category.
\end{thm}

\begin{proof}
The subcategory $\Add\X=\Add X$ consists only of endofinite objects. As these are necessarily $\Sigma$-pure injective by Proposition~\ref{pr:subgroup-fin-def} we see $\PExt^1(\X,\Add\X)=0$. Also, $\Add X$ is closed under pure subobjects by Remark~\ref{re:quot-lattice}.

Next, writing $J$ for the Jacobson radical of $\End(X)$, Nakayama's Lemma implies that, for each $C\in\fp\A$, there exists an $n$ such that $J^n\Hom(C,X)=0$. It follows that $\X$ is locally T-nilpotent, so $\Add\X$ equals $\Coprod\X$ and is closed under filtered colimits. 

Finally, let $U$ be an $\X$-basic subobject of $X^I$ containing $X^{(I)}$ as a direct summand. As the characters of $X$, $X^{(I)}$ and $X^I$ agree we must have $\character_U=\character_X$. If the complement to $U$ is non-zero, then it contains an indecomposable summand $V$, and again $\character_{U\oplus V}=\character_X$. On the other hand, as $\End(V)$ is local with nilpotent radical, we see that $fg$ is nilpotent for all morphisms $f\colon U\to V$ and $g\colon V\to U$. It follows that both $\Hom(U,V)$ and $\Hom(V,U)$ are contained in the Jacobson radical of $\End(U\oplus V)$. We deduce that there is an equivalence between the categories of semisimple modules over $\End(U\oplus V)$ and over $\End(U)\times\End(V)$. Applying this to the socle filtration of $\Hom(C,U\oplus V)$ shows that $\character_{U\oplus V}=\character_U+\character_V$, a contradiction. Thus $X^I=U\in\Add\X$.
\end{proof}

\bibliographystyle{plain}
\bibliography{refs}
\end{document}